\RequirePackage{fix-cm}
\RequirePackage{amsmath}
\documentclass[smallextended]{svjour3}       

\smartqed  

\usepackage[english]{babel}
\usepackage{mathtools}
\usepackage{amssymb}
\usepackage{graphicx}
\usepackage{stmaryrd}
\usepackage{bm}
\usepackage{tikz-cd}
\usepackage{enumitem}
\tikzcdset{ arrow style= math font }
\usepackage[hidelinks]{hyperref}
\numberwithin{equation}{section}

\newcommand{\inpo}[2]{\left(#1,\, #2\right)_\Omega}
\newcommand{\inpg}[2]{\left(#1,\, #2\right)_\Gamma}
\newcommand{\asym}{{\operatorname{asym}\,}}
\newcommand{\skw}{{\operatorname{skw}}}
\newcommand{\sym}{{\operatorname{sym}\,}}

\newcommand{\Tr}{{\operatorname{Tr}\,}}
\newcommand{\curl}{{\operatorname{curl}}}

\renewcommand{\div}{{\operatorname{div}}}

\newcommand{\enorm}[1]{{\left\vert\kern-0.25ex\left\vert\kern-0.25ex\left\vert #1 
  \right\vert\kern-0.25ex\right\vert\kern-0.25ex\right\vert}}
\newcommand{\normL}[1]{\| #1 \|_{\Omega} }
\newcommand{\normLg}[1]{\| #1 \|_{\Gamma} }
\newcommand{\norm}[1]{\| #1 \| }
\newcommand{\jump}[1]{\llbracket #1 \rrbracket }
\newcommand{\hatj}{{\hat{\jmath}}}

\newlength{\arrow}
\begin{document}

\title{Stable Mixed Finite Elements for Linear Elasticity with Thin Inclusions 
	\thanks{The research of the authors was funded in part by the Norwegian Research Council
grants 233736, 250223. The first author thanks the German Research Foundation (DFG) for supporting this work by funding SFB 1313, Project Number 327154368.}
	}
\author{W. M. Boon \and J. M. Nordbotten}

\institute{W. M. Boon \at              
	Institute for Modelling Hydraulic and Environmental Systems, 
	University of Stuttgart,
	70569 Stuttgart, Germany \\
	\email{wietse.boon@iws.uni-stuttgart.com}
           \and
           J. M. Nordbotten \at
           Department of Mathematics, 
           University of Bergen, 
           5020 Bergen, Norway
}

\date{Received: date / Accepted: date}

\maketitle

\begin{abstract}
	We consider mechanics of composite materials in which thin inclusions are modeled by lower-dimensional manifolds. By successively applying the dimensional reduction to junctions and intersections within the material, a geometry of hierarchically connected manifolds is formed which we refer to as mixed-dimensional. 

	The governing equations with respect to linear elasticity are then defined on this mixed-dimensional geometry. The resulting system of partial differential equations is also referred to as mixed-dimensional, since functions defined on domains of multiple dimensionalities are considered in a fully coupled manner. With the use of a semi-discrete differential operator, we obtain the variational formulation of this system in terms of both displacements and stresses. The system is then analyzed and shown to be well-posed with respect to appropriately weighted norms.

	Numerical discretization schemes are proposed using well-known mixed finite elements in all dimensions. The schemes conserve linear momentum locally while relaxing the symmetry condition on the stress tensor. Stability and convergence are shown using a priori error estimates.

\end{abstract}

\section{Introduction}

Thin inclusions in elastic materials arise in a variety of scientific fields, including geo-physics, bio-mechanics, and the study of composite materials. The subsurface, for example, typically includes rock layers with significantly larger horizontal extent compared to their height. Since it is often infeasible to resolve such small heights for large-scale simulations, we consider the setting where the layer, or aquifer, is represented by a lower-dimensional manifold. The governing equations on this manifold can be derived using vertical integration \cite{nordbottenbook}, see e.g. \cite{bjornaraa2016vertically} for an application with respect to CO$_2$ storage.

Secondly, membranes occur frequently in the study of bio-mechanical systems. Examples range from cell walls in plants to the heart sac and dermal layer in human physiology. As a modeling assumption, each of these membranes can be represented by lower-dimensional manifolds. Their influence on the coupled mechanical system can then be incorporated by assigning significantly different material properties compared to the surroundings.

A third application concerns the study of composite or reinforced materials. In this context, the lower-dimensional manifolds correspond to the stiffer plates embedded in the material for strengthening purposes. This can be expanded to connected, two-dimensional objects such as H-beams and T-beams. The junctions are then considered one-dimensional manifolds, with inherited or separately defined material properties. We note that this work is limited to manifolds of codimension one and thus does not treat the case of embedded, one-dimensional rods in three dimensions.

The thin features are considered lower-dimensional and have elastic properties, yet a slightly different setting is presented than the conventional theory of thin shells \cite{ciarlet2000theory}. The main difference is that we focus on a strong coupling of a thin inclusion with a surrounding, elastic medium. The interest of this work is therefore more closely aligned to \cite{caillerie1980effect}, in which rigid, thin inclusions are considered.

The structure of the derived equations fits well with the mixed-dimensional framework derived in \cite{Nordbotten2017DD,boon2017excalc}. We aim to preserve this structure and retain a local conservation of linear momentum after discretization with the use of conforming, mixed finite elements. The construction of stable finite element pairs representing displacements and symmetric stresses is involved, and typically leads to higher-order elements for the stress space \cite{arnold2002mixed}. By relaxing the symmetry condition on the stress tensor as in \cite{arnold2006differential,awanou2013rectangular}, these difficulties can be mitigated.

This article is structured as follows. Section~\ref{sec:geometric_representation} introduces the notational conventions and the decomposition of the geometry according to dimension. On this geometry, Section~\ref{sec: Model Formulation} introduces the governing equations of the model in all dimensions. We introduce the relevant function spaces and present the derivation of the variational formulation in Section~\ref{sec:weak_form}. The resulting system of equations is proven to be well-posed in Section~\ref{sec: Well-posedness}. Finally, we propose conforming discretization schemes in Section~\ref{sec: Discretization} for which stability and convergence are shown.

\section{Geometry and Notation}
\label{sec:geometric_representation}

In this section, we introduce the mixed-dimensional geometry and establish notation. Here, we follow the conventions introduced in \cite{Boon2018Robust,boon2017excalc}.

Let us consider an $n$-dimensional domain $Y$ that contains thin, embedded structures, represented by lower-dimensional manifolds. In general, we consider $n = 3$, the two-dimensional case being simpler. Let $\Omega_i^{d_i}$ be such a manifold, with $i$ the unique index from a global set $I$ and $d_i$ its dimension. The superscript is frequently omitted for brevity. For $1 \le d \le n - 1$, we successively identify intersections between $d$-manifolds as $(d - 1)$-manifolds. All $\Omega_i$ are open sets and mutually disjoint. For simplicity, we restrict this work to the case in which all $\Omega_i$ have zero curvature, i.e. are flat. 

As an example, let us consider the two-dimensional set-up in Figure~\ref{fig: domain decomp} (left). Here, an embedded H-beam is described using two zero-dimensional intersection points and five one-dimensional line segments. The open set corresponding to the surrounding medium is given by $\Omega_8^2 = \operatorname{int}(Y \setminus \cup_{i = 1}^7 \Omega_i^{d_i})$.

We refer to the union $\bigcup_{i \in I} \Omega_i$ as the mixed-dimensional geometry $\Omega$. Let $I^d$ be the set of indices corresponding to $d$-manifolds and let $\Omega^d$ be the collection of such manifolds. In short, we denote
\begin{align*}
		I^d &= \{ i \in I: \ d_i = d \}, &
		\Omega^d &= \bigcup_{i \in I^d} \Omega_i, &
		\Omega &= \bigcup_{d = 0}^n \Omega^d.
\end{align*}

The interface between manifolds of codimension one will play an important role, and we adopt a separate notation for these. Let $J$ be the set of indices such that each $j \in J$ corresponds to an interface $\Gamma_j$ between $\Omega_i$ for some $i \in I$ and an adjacent domain of dimension $(d_i+1)$. $\Gamma_j$ physically coincides with $\Omega_i$ and we assume that a unique $\hatj \in I$ exists such that $\Gamma_j \subseteq \partial \Omega_\hatj$. 

To distinguish these different interfaces, we define the following index sets for $i \in I$:
\begin{align*}
	\hat{J}_i &= \{ j \in J: \ \Gamma_j = \Omega_i \} &
	\check{J}_i &= \{ j \in J: \ \Gamma_j \subseteq \partial \Omega_i \}.
\end{align*}
An example is shown in Figure~\ref{fig: domain decomp} (right), which emphasizes that for $j_1, j_2 \in \hat{J}_4$ with $j_1 \ne j_2$, we have $\hatj_1 = \hatj_2 = 8$. In other words, we allow for a manifold $\Omega_\hatj$ to border on multiple sides of $\Omega_i$ and assign a unique index $j \in \hat{J}_i$ to each side. Finally, we remark that $\hat{J}_i$ is void for all $i \in I^n$, by definition.

Using the same summation convention per dimension as above, we denote
	\begin{align*}
		\Gamma^d &= \bigcup_{i \in I^d} \bigcup_{j \in \hat{J}_i} \Gamma_j, &
		\Gamma &= \bigcup_{d = 0}^{n - 1} \Gamma^d.
	\end{align*}
Each $\Gamma_j$ is equipped with a unit normal vector $\bm{n}_j$, from the tangent space of $\Omega_\hatj$, oriented outward with respect to $\Omega_\hatj$. The subscript on $\bm{n}$ is omitted for brevity. In reference to the vector(s) normal to a manifold $\Omega_i$, we will use the check notation $\check{\bm{n}}_i$.

\begin{figure}[ht]
	\centering
	\includegraphics[width = 0.9 \textwidth]{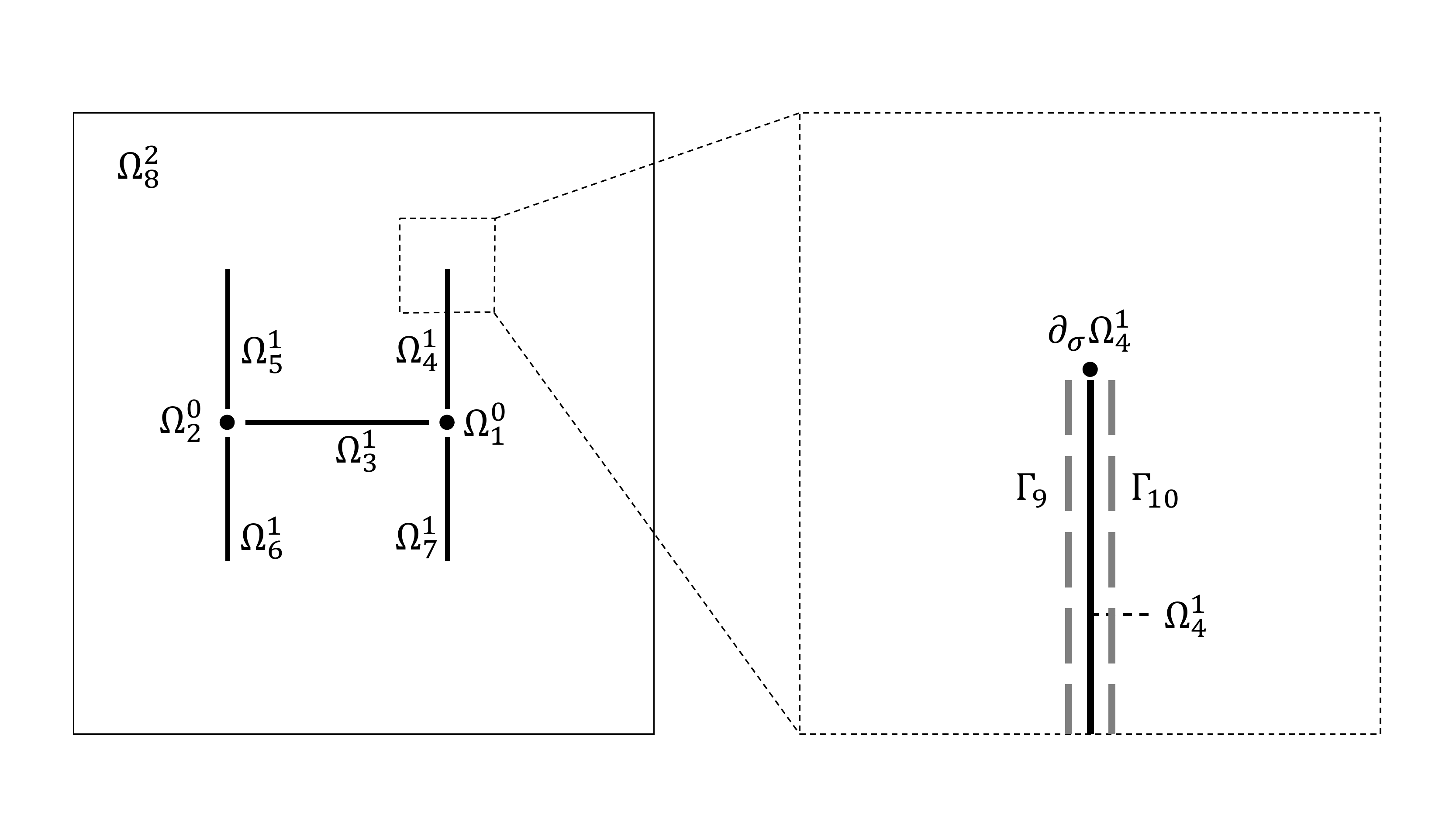}
	\caption{The domain is decomposed into manifolds of different dimensionalities $\Omega_i^d$ with $i$ the global index and $d = d_i$ its dimension. (Left) The intersection points become zero-dimensional manifold in the decomposition. Here, the index sets are given by $I^0 = \{ 1, 2 \}$, $I^1 = \{ 3, 4, 5, 6, 7 \}$, and $I^2 = \{ 8 \}$. (Right) A zoom on one of the extremities showing the logical interpretation of $\Omega_4^1$ and its interfaces ($\Gamma_j$) with $\Omega_8^2$. In this case, we have $\hatj = 8$ for all $j \in \hat{J}_4 = \{ 9, 10 \}$ and for $j_1, j_2, j_3 \in \hat{J}_1$, we have $(\hatj_1, \hatj_2, \hatj_3) = (3, 4, 7)$. On the extremity $\partial_\sigma \Omega_4^1$, a zero stress condition is imposed.}
	\label{fig: domain decomp}
\end{figure}

The boundary of the domain is given by the disjoint union $\partial_\sigma Y \bigcup \partial_u Y$ on which different boundary conditions will be imposed. In particular, we assume that the displacement is given on $\partial_u Y$ and the normal stress on $\partial_\sigma Y$. We denote for $i \in I$,
\begin{align*}
	\partial_u \Omega_i &= \partial_u Y \cap \partial \Omega_i, &
	\partial_\sigma \Omega_i &= \partial \Omega_i \setminus 
	\left(\cup_{j \in \check{J}_i} \Gamma_j \cup \partial_u \Omega_i \right).
\end{align*}
For analysis purposes, we assume that $|\partial_u \Omega_i| > 0$ for all $i \in I^n$, i.e. each subdomain of dimension $n$ is connected to a part of the boundary on which the displacement is prescribed. By omission of the subscript, we use $\partial_u \Omega$ and $\partial_\sigma \Omega$ to refer to the corresponding boundaries of the mixed-dimensional geometry.

Given a function $f$ defined on the mixed-dimensional geometry $\Omega$, let $f_i$ denote its restriction to $\Omega_i$, i.e. $f_i = (f)|_{\Omega_i}$. Furthermore, we employ the hat and check notation to distinguish instances of $f$ inherited from different domains onto the interface $\Gamma$. This means that on $\Gamma_j$ with $j \in \hat{J}_i$, we denote
\begin{align*}
	\check{f} &:= f_i, &
	\hat{f} &:= (f_\hatj)|_{\Gamma_j}.
\end{align*}
Note that the definition of $\hat{f}$ involves a trace of $f_\hatj$ onto $\Gamma_j$.

\section{Model Formulation} 
\label{sec: Model Formulation}

In this section, we consider the governing equations and introduce the model problem. Starting with the mathematical formulation of linear elasticity in the surrounding medium, we continue with the generalized equations on lower-dimensional manifolds to derive the strong form of the mechanics problem. The variational formulation is considered afterward in Section~\ref{sec:weak_form}.

\subsection{Governing Equations in the Surrounding Medium}
\label{sub:matrix equations}

Let us start by presenting the governing equations for linear elasticity in the surrounding medium $\Omega^n$. For $i \in I^n$, let $\sigma_i$ denote the elastic stress and $\bm{u}_i$ the displacement. Assuming infinitesimal strain, the stress-strain relationship has the general form:
\begin{align*}
	A \sigma_i &= \varepsilon(\bm{u}_i) = \sym(\nabla \bm{u}_i) = \frac{1}{2}\left(\nabla \bm{u}_i + (\nabla \bm{u}_i)^T\right). 
\end{align*}
In case of homogeneous and isotropic media, the operator $A$ describes Hooke's law and is given by
\begin{align}
	A \sigma_i  = \frac{1}{2\mu}\left(\sigma_i - \frac{\lambda}{2\mu + n\lambda} \Tr \sigma_i I \right), \label{operator A}
\end{align}
in which $\lambda$ and $\mu$ are the Lam\'e parameters and Tr is the matrix trace operator.

In the variational formulation presented Section~\ref{sec:weak_form}, the symmetry of the stress tensor $\sigma_i$ will be enforced in a weak sense. In preparation, we introduce the antisymmetric tensor $\chi_i = \asym{\nabla \bm{u}_i}$ such that:
\begin{align*}
  A\sigma_i &= \nabla \bm{u}_i - \chi_i.
\end{align*}

With the addition of linear and angular momentum conservation, the following system of equations is formed in each $\Omega_i$ with $i \in I^n$.
\begin{subequations} \label{strong system n}
	\begin{align}
		A\sigma_i - \nabla \bm{u}_i + \chi_i &= 0, \label{stress-strain n}\\
		\nabla \cdot \sigma_i &= \bm{f}_i, \label{mom balance n}\\
		\asym{\sigma_i} &= 0. \label{ang mom balance n}
	\end{align}
With $\bm{f}_i$ the body forces acting on $\Omega_i$. Note that the balance of angular momentum is enforced as the symmetry of the stress tensor $\sigma_i$ in \eqref{ang mom balance n}. The associated boundary conditions are given by
\begin{align}
	\bm{n} \cdot \sigma_i &= 0 \text{ on } \partial_\sigma \Omega_i, & 
	\bm{u}_i &= \bm{g}_u \text{ on } \partial_u \Omega_i.
\end{align}
\end{subequations}
with $\bm{g}_u$ a given function. We limit the exposition to homogeneous stress boundary conditions, noting that this can readily be extended to the general case.

\subsection{Geometrical Scaling and Constraints}
\label{sub:geometrical_scaling_and_constraints}

The governing equations on the lower-dimensional manifolds will be significantly influenced by the small width of the thin inclusions. We therefore devote this section to defining two key parameters, $\gamma$ and $\epsilon$ and the constraints these parameters adhere to.

Let $\gamma$ be a virtual parameter representing the relatively small length from the interface of the higher-dimensional domain to the central plane (2D), line (1D), or point (0D) of the physical inclusion. We assume that on each interface $\Gamma_j$, $\gamma$ is constant and positive.

The small value of $\gamma$ will introduce a scaling in the equations. In the formulation of the problem, it is advantageous if this scaling appears in the coupling terms between variables, rather than on the main diagonal of the system. For this purpose, we introduce a key concept from the context of subsurface flow models \cite{arbogast2016linear,Boon2018Robust}. As shown in those works, the desired scaling of the system is achieved by employing an appropriately scaled flux variable. In analogy, we employ a scaled stress by introducing the scaling parameter $\epsilon$ on $\Omega$. This parameter is defined as the square root of the (small) cross-sectional length (2D), area (1D), or volume (0D) of the dimensionally reducible feature. $\epsilon$ is assumed to be constant and positive for each manifold $\Omega_i$. The following relationship is imposed
\begin{align*}
	V_i &= \epsilon_i^2 \eqsim \gamma_j^{n - d_i}, & 
	\forall i &\in I, \ j \in \hat{J}_i.
\end{align*}
with $V_i$ the cross-sectional measure of the feature corresponding to $\Omega_i$ and $\epsilon_i = 1$ for $i \in I^n$. The relation $a \eqsim b$ (respectively $\lesssim$ and $\gtrsim$) implies that a constant $C>0$ exists, independent of $\epsilon$, $\gamma$, and the mesh size $h$, such that $a = C b$ (respectively $\le$ and $\ge$). Since $\epsilon$ is assumed to be small, we often use the relationship $\epsilon \lesssim 1$. 

Next, we relate the different values of $\epsilon$ between the dimensions. For each $\Omega_i$, let $\hat{\epsilon}_{\max}$ be the maximal value of $\epsilon$ in the adjacent higher-dimensional manifolds:
	\begin{align}
		(\hat{\epsilon}_{\max})|_{\Omega_i} &:= \max_{j \in \hat{J}_i} \ (\hat{\epsilon})|_{\Gamma_j} > 0, 
		& \forall i &\in I. \label{eq: def epsmax}
	\end{align}
In $\Omega^n$, we define $\hat{\epsilon}_{\max}$ as unity. Using this definition, we add a constraint to the geometry by assuming that $\hat{\epsilon}_{\max}$ bounds $\epsilon$ from above, i.e.
	\begin{align}
		\epsilon &\lesssim \hat{\epsilon}_{\max}. \label{eps bounded by emax}
	\end{align}

With the parameter $\epsilon$ defined, we continue with the scaling of the stress variable on the lower-dimensional manifolds. In $\Omega_i$ with $d_i < n$, let $\sigma^{avg}_i$ contain the columns of the Cauchy stress tensor associated with the tangent bundle of $\Omega_i$, averaged over the cross-section of the physical feature. The integrated stress tensor $\sigma^{int}_i$, on the other hand, is obtained after multiplication with the cross-sectional measure $V_i$. 

Using the factor $\epsilon_i := \sqrt{V_i}$, the scaled stress $\sigma_i$ on $\Omega_i$ is defined as
\begin{align*}
	\epsilon \sigma^{avg}_i = \sigma_i = \epsilon^{-1} \sigma^{int}_i.
\end{align*}

The columns and first $d_i$ rows of $\sigma_i$ correspond to the basis vectors from the tangent bundle of $\Omega_i$. The final $(n - d_i)$ rows relate to the directions normal to the manifold. Thus, $\sigma_i \in \mathbb{R}^{n \times d_i}$ by definition, making it undefined in the intersection points $\Omega^0$. The displacement $\bm{u}_i$ in $\Omega_i$ remains unscaled and is naturally in $\mathbb{R}^n$. Again, we omit the subscript $i$ to refer to the mixed-dimensional entities, i.e.
\begin{align*}
 	\sigma &= \bigoplus_{d = 1}^n \bigoplus_{i \in I^d} \sigma_i, &
 	\bm{u} &= \bigoplus_{i \in I} \bm{u}_i.
\end{align*}
We emphasize that the integrated and average stress quantities can readily be recovered from the scaled stress $\sigma$ after appropriate post-processing with the known quantity $\epsilon$. 

\subsection{Mixed-Dimensional Equations}
\label{sub:embedded_equations}

With the given scaling from the previous subsection, let us consider the governing equations in the lower-dimensional manifolds. In this generalization to the mixed-dimensional geometry, a structure similar to the system \eqref{strong system n} is uncovered. We start by introducing the linear momentum balance equation, followed by the stress-strain relationships and finish with the conservation of angular momentum.

The balance of linear momentum \eqref{mom balance n} is generalized first. After integrating the conservation law in the direction(s) normal to the inclusion (see e.g. \cite{Boon2018Robust,Roberts2} for the analogue in fracture flow models), we obtain 
\begin{align*}
	\nabla \cdot \epsilon_i \sigma_i - \sum_{j \in \hat{J}_i} (\bm{n} \cdot \epsilon_\hatj \sigma_\hatj)|_{\Gamma_j} &= \epsilon_i^2 \bm{f}_i, 
	& \text{in } \Omega_i, \ i \in \bigcup_{d = 1}^{n - 1} I^d.
\end{align*}
Here, $\bm{f}_i$ is the body force acting on $\Omega_i$, averaged over the cross-section with measure $\epsilon^2$. The in-plane divergence $(\nabla \cdot)$ on $\Omega_i$ and the normal trace operator $(\bm{n} \cdot)$ onto $\Gamma_j$ are applied row-wise. Hence, the divergence in $\Omega_i$ maps from $\mathbb{R}^{n \times d_i}$ to $\mathbb{R}^n$ and the normal trace on $\Gamma_j$ maps $\mathbb{R}^{n \times d_\hatj}$ to $\mathbb{R}^n$.

For the zero-dimensional manifolds, there are no divergence operator or $\sigma_i$ available and the balance law is completely given by the sum of forces from $\Omega^1$:
\begin{align*}
	\sum_{j \in \hat{J}_i} (\bm{n} \cdot \epsilon_\hatj \sigma_\hatj)|_{\Gamma_j} &= \epsilon_i^2 \bm{f}_i, 
	& \text{in } \Omega_i,\ i \in I^0.
\end{align*}

To shorten notation, we introduce the jump operator $\jump{\cdot}_i$ which maps functions defined on the interface $\Gamma_j$ with $j \in \hat{J}_i$ to the central manifold $\Omega_i$ such that
\begin{align*}
	\jump{\phi}_i &= \sum_{j \in \hat{J}_i} \phi|_{\Gamma_j}, & \forall i &\in I.
\end{align*}

Following \cite{Boon2018Robust}, we introduce the mixed-dimensional divergence operator ($\mathfrak{D} \cdot$) as $\mathfrak{D} \cdot \sigma := \nabla \cdot \sigma - \jump{ \bm{n} \cdot \hat{\sigma} }$ and rewrite the conservation equation to the concise form
\begin{align}
	\mathfrak{D} \cdot \epsilon \sigma &= \epsilon^2 \bm{f}. \label{mom balance manifold}
\end{align}

We now continue by defining the stress-strain relationships in the lower-dimensional manifolds in analogy with \eqref{stress-strain n}. For that, we first introduce the gradient operator $\mathfrak{D}$ as
\begin{align*}
	\mathfrak{D} \bm{u}
	&= \left\{\begin{aligned}
	& \nabla \bm{u}, &\text{in } &\Omega, \\
	& \check{\bm{u}} - \hat{\bm{u}}, &\text{on } &\Gamma.
		\end{aligned} \right.
\end{align*}
We emphasize that the gradient $\nabla$ relates to the tangential direction(s) and is applied row-wise. Since we have $\mathfrak{D}\bm{u}$ defined on both $\Omega$ and $\Gamma$, we need to provide stress-strain relationships inside and on the boundaries of the domains. 

The stress-strain relationship is then described by an operator $\mathfrak{A}$ acting on the averaged stress. Here, we pay special attention to the scaling with $\epsilon$. Thus, recalling that the averaged stress is denoted by $\sigma^{avg} = \epsilon^{-1} \sigma$, the stress-strain relationships are given by
\begin{align*}
	\mathcal{A} (\epsilon^{-1} \sigma)
	&= \mathfrak{D} \bm{u} - \chi, &
	\text{on }& \Omega \times \Gamma
\end{align*}
with $\chi$ to be defined.
To obtain a symmetric system, we scale this equation with $\epsilon$. Noting that $\epsilon$ and $\mathcal{A}$ do not necessarily commute, we introduce $\mathfrak{A} := \epsilon \mathcal{A} \epsilon^{-1}$ to obtain the generalized version of the stress-strain relationship:
\begin{align}
	\mathfrak{A} \sigma
	&= \epsilon \mathfrak{D} \bm{u} - \epsilon \chi, &
	\text{on }& \Omega \times \Gamma.
	\label{stress-strain tilde}
\end{align}
The restrictions of $\mathfrak{A}$ to the manifolds $\Omega$ and interfaces $\Gamma$ are respectively denoted by
\begin{align*}
	\mathfrak{A}_\| \sigma &:= (\mathfrak{A} \sigma)|_\Omega, &
	\mathfrak{A}_\perp \sigma &:= (\mathfrak{A} \sigma)|_\Gamma
\end{align*}
The variable $\chi$ is the generalization of the asymmetric $\chi_i$ from section~\ref{sub:matrix equations}, given by
\begin{align*}
	\chi_i &= \begin{bmatrix}
		\asym (\nabla \bm{u}_{i, \|}) \\
		0
	\end{bmatrix} \in \mathbb{R}^{n \times d_i}, &\text{in } \Omega_i,\ i \in \bigcup_{d = 2}^{n} I^d.
\end{align*}
We interpret $(\chi)|_\Gamma = 0$ and $\chi_i = 0$ for $i \in I^0 \cup I^1$.

\begin{example}
	We provide an explicit example of $\mathfrak{A}$ using a fictitious material. In this material, we assume that the stress-strain relationships in tangential and normal directions are independent. This assumption leads to a model which captures in-plane shearing whereas out-of-plane stress components follow a one-dimensional Hooke's law. This behavior is described by the following constitutive laws
\begin{subequations}\label{eq: stress-strain tangential}
\begin{align} 
	\epsilon^{-1} \sigma &= 
	\begin{bmatrix}
		2\mu \sym(\nabla \bm{u}_{\|}) + \lambda \Tr (\nabla \bm{u}_{\|}) I \\
		2\mu \nabla \bm{u}_{\perp}
	\end{bmatrix}, &\text{in } \Omega, \\
	\bm{n} \cdot (\hat{\epsilon}^{-1} \hat{\sigma}) 
	&= 2\mu_\perp (\check{\bm{u}} - \hat{\bm{u}}) + \lambda_\perp (\bm{n} \cdot (\check{\bm{u}} - \hat{\bm{u}})) \bm{n}, &\text{on } \Gamma,
\end{align}\end{subequations}
Here, $\mu$ and $\lambda$ (respectively $\mu_\perp$ and $\lambda_\perp$) are the Lam\'e parameters describing the stress-strain relationship tangential (and normal) to the manifold. 

The inverse relations, mapping stresses to strains, are then given by
\begin{align*}
	(\mathfrak{A}_\| \sigma)|_{\Omega_i}
	&= 
	(2\mu)^{-1}\left( \sigma_i - \lambda (2\mu + d_i\lambda)^{-1} \Tr \left(\sigma_i \right) [I_{d_i}, 0]^T \right), 
	&i &\in \bigcup_{d = 1}^n I^d, \\
	(\mathfrak{A}_\perp \sigma)|_{\Gamma_j}
	&= (2 \mu_\perp)^{-1} (\bm{n} \cdot \sigma_\hatj - \lambda_\perp (2 \mu_\perp + \lambda_\perp)^{-1} (\bm{n} \cdot \sigma_\hatj \cdot \bm{n}) \bm{n}), 
	& j &\in J.
\end{align*}
Here, $I_d$ is the identity tensor in $\mathbb{R}^{d \times d}$. We remark that in this example, we have $\epsilon \mathfrak{A} \epsilon^{-1} = \mathfrak{A}$.
\qed
\end{example}

Finally, we consider the symmetry of the stress tensor. 
Since the lower-dimensional manifolds model objects with finite width, the limit argument used to prove symmetry of the stress tensor is only valid within manifolds (and not transversely). Consequently, symmetry of the stress tensor is imposed within each manifold, expressed as: 
\begin{align}
	\asym \epsilon \sigma &= 0.
	\label{eq: sym manifold}
\end{align}
We remark that for $i \in I^0 \cup I^1$, this equation is trivial since either $\sigma_i$ does not exist or is a vector. For $i \in I^2$, this equation evaluates the asymmetry of the in-plane components.

Gathering \eqref{mom balance manifold}, \eqref{stress-strain tilde}, and \eqref{eq: sym manifold}, we arrive at the strong form of the generalized system of equations:
\begin{subequations} \label{strong system fracture}
	\begin{align}
 		\mathfrak{A} \sigma
		- \epsilon \mathfrak{D} \bm{u}
		+ \epsilon \chi
		&= 0 & \text{ in }&\Omega \times \Gamma, \label{s-s}\\
		\mathfrak{D} \cdot \epsilon \sigma &= \epsilon^2 \bm{f} & \text{ in }&\Omega, \\
		\asym \epsilon \sigma &= 0 & \text{ in }&\Omega.
	\end{align}
We emphasize that $\epsilon$ is defined as the square root of the cross-sectional measure, leading to the appearance of $\epsilon^2$ in the second equation. To close the system, the boundary conditions are given by
\begin{align}
	\bm{n} \cdot \epsilon \sigma &= 0 \text{ on } \partial_\sigma \Omega, & 
	\bm{u} &= \bm{g}_u \text{ on } \partial_u \Omega.
\end{align}
\end{subequations}
System \eqref{strong system fracture} has a structure similar to \eqref{strong system n} in that it
is composed of constitutive law(s) complemented with a differential and algebraic constraint. This
structure is common in models concerning linear elasticity with relaxed symmetry
\cite{arnold2006differential,awanou2013rectangular}. We will show in the next section that the
system indeed corresponds to a symmetric saddle-point problem.

\section{Variational Formulation}
\label{sec:weak_form}

With the goal of obtaining a mixed finite element discretization, this section presents the weak formulation of the continuous problem. In order to do this, we introduce several analytical tools. First, the relevant function spaces are defined as well as the notational conventions concerning inner products. Next, we derive the variational formulation of \eqref{strong system fracture} and show that it corresponds to a symmetric saddle point problem. 

\subsection{Function Spaces}

The function spaces relevant for this problem are constructed as products of familiar function spaces on the $d$-dimensional manifolds. In particular, we define
\begin{subequations} \label{cont spaces}
	\begin{align}
			\Sigma &= \prod_{d = 1}^n \prod_{i \in I^d}
				\left\{\tau_i \in (H(\div, \Omega_i))^n : \ 
			\begin{aligned}
				\bm{n} \cdot \tau_i&|_{\partial_\sigma \Omega_i} = 0, \\
			 	\bm{n} \cdot \tau_i&|_{\Gamma_j} \in (L^2(\Gamma_j))^n, \ \forall j \in \check{J}_i
			\end{aligned}
				\right\} \\
			\bm{U} &= \prod_{d = 0}^n \prod_{i \in I^d}
			\left(L^2(\Omega_i)\right)^n, \\
			R &= \prod_{d = 2}^n \prod_{i \in I^d}
			\left(L^2(\Omega_i)\right)^{k_d},
	\end{align}
\end{subequations}
where $\Sigma$ denotes the function space for the stress, $\bm{U}$ contains the displacement, and $R$ is the function space for the Lagrange multiplier enforcing symmetry of the stress tensor. The exponent $k_d$ is given by
$k_d = \left(\begin{smallmatrix} d \\ 2 \end{smallmatrix}\right) = d (d - 1) / 2$
, see e.g. \cite{arnold2006differential,awanou2013rectangular}. 

The mixed-dimensional $L^2$-inner products on $\Omega$ and $\Gamma$ are defined as the sum of inner products over all corresponding manifolds:
\begin{align*}
	(f, g)_{\Omega} &= \sum_{i \in I} (f_i, g_i)_{\Omega_i} 
	, & 
	(\phi, \varphi)_{\Gamma} &= \sum_{i \in I} \sum_{j \in \hat{J}_i} (\phi_i, \varphi_i)_{\Gamma_i} 
\end{align*}
Here, the implicit assumption is made that the contribution is zero for all manifolds on which $f$ is undefined. For example, for $\sigma, \tau \in \Sigma$, the inner product $(\sigma, \tau)_{\Omega}$ has no contribution on $\Omega^0$. Likewise for functions in $R$, the inner product is zero on manifolds $\Omega_i$ with $i \in I^0 \cup I^1$.

For functions $\sigma, \tau \in \Sigma$, we note that they are defined on both $\Omega$ and $\Gamma$. For convenience, we introduce the combined inner product
\begin{align*}
	(\sigma, \tau)_{\Omega \times \Gamma} &:= (\sigma, \tau)_{\Omega} + (\bm{n} \cdot \sigma, \bm{n} \cdot \tau)_{\Gamma},
\end{align*}
which, in the case of the operator $\mathfrak{A}$, is understood as
\begin{align*}
	(\mathfrak{A} \sigma, \tau)_{\Omega \times \Gamma} 
	&:= (\mathfrak{A}_\| \sigma, \tau)_{\Omega} 
	+ (\mathfrak{A}_\perp \sigma, \bm{n} \cdot \tau)_{\Gamma}.
\end{align*}

The inner products naturally induce the $L^2$-type norms $\| \cdot \|_{\Omega}$, $\| \cdot \|_{\Gamma}$, and $\| \cdot \|_{\Omega \times \Gamma}$. With these norms, we assume that $\mathfrak{A}$ is continuous and coercive with respect to the norm $\| \cdot \|_{\Omega \times \Gamma}$. Thus, for all $\sigma, \tau \in \Sigma$, we have: 
		\begin{align}\label{Bounds A}
			(\mathfrak{A} \sigma, \tau)_{\Omega \times \Gamma} &\lesssim 
			\norm{\sigma}_{\Omega \times \Gamma}
			\norm{\tau}_{\Omega \times \Gamma} &
			(\mathfrak{A} \sigma, \sigma)_{\Omega \times \Gamma} &\gtrsim 
			\norm{\sigma}_{\Omega \times \Gamma}^2
		\end{align} 
We emphasize that the constants within these bounds are independent of $\epsilon$.

\subsection{Identifying the Symmetric Saddle Point Problem}

In this section, we make two key observations which allow us to derive a variational formulation of \eqref{strong system fracture} which is symmetric. First let us consider the terms containing $u$ in the stress-strain relationships \eqref{s-s} and \eqref{s-s}. We multiply these terms with $\tau \in \Sigma$ and $\bm{n} \cdot \tau \in L^2(\Gamma)$, respectively, and integrate to obtain the following integration by parts formula:
\begin{align}
	(\epsilon \mathfrak{D}\bm{u}, \tau)_{\Omega \times \Gamma} 
	&= \inpo{\check{\epsilon} \nabla \bm{u}}{\tau} +  \inpg{\hat{\epsilon} (\check{\bm{u}}- \hat{\bm{u}})}{\bm{n} \cdot \hat{\tau}} \nonumber \\
	&= - \inpo{\bm{u}}{\nabla \cdot \check{\epsilon} \tau}
	+ (\bm{u}, \bm{n} \cdot \epsilon \tau)_{\partial_u \Omega}
	+ \inpg{\hat{\bm{u}}}{\bm{n} \cdot \hat{\epsilon} \hat{\tau}} \nonumber \\
	&\ \ \ + \inpg{\hat{\epsilon} \check{\bm{u}}}{\bm{n} \cdot \hat{\tau}}
	- \inpg{\hat{\epsilon} \hat{\bm{u}}}{\bm{n} \cdot \hat{\tau}} \nonumber \\
	&= - \inpo{\bm{u}}{\nabla \cdot \check{\epsilon} \tau}
	+ \inpo{\bm{u}}{\jump{\bm{n} \cdot \hat{\epsilon} \hat{\tau}}} 
	+ (\bm{u}, \bm{n} \cdot \epsilon \tau)_{\partial_u \Omega} \nonumber \\
	&= - \inpo{\bm{u}}{\mathfrak{D} \cdot \epsilon \tau}
	+ (\bm{u}, \bm{n} \cdot \epsilon \tau)_{\partial_u \Omega}. \label{eq: integration by parts}
\end{align}
Here $\mathfrak{D} \cdot$ is the mixed-dimensional divergence operator from \eqref{mom balance manifold}.

Secondly, we introduce the operator $\skw$ which evaluates the asymmetric part of a matrix. More specifically, for a matrix $B \in \mathbb{R}^{n \times d}$ with components $b_{ij}$, let
\begin{align*}
	\skw B  &=
	\left\{
	\begin{aligned}
		[b_{23} - b_{32},\ b_{31} - b_{13},\ & b_{12} - b_{21}]^T, & d = 3, \\
		& b_{12} - b_{21}, & d = 2.
	\end{aligned}
	\right.
\end{align*}
This operator is naturally lifted to $\skw: \Sigma \to R$. Next, we turn our attention to the term in \eqref{s-s} containing the asymmetric variable $\chi$. Let us multiply this term with $\tau \in \Sigma$ and integrate over $\Omega^2 \cup \Omega^3$. With the introduction of $r = \frac{1}{2} \skw \chi \in R$, we obtain
\begin{align*}
	\inpo{ \epsilon \chi}{\tau}
	&= \inpo{\chi}{ \epsilon \tau}
	= \inpo{r}{\skw  \epsilon \tau}.
\end{align*}

By employing test functions $(\tau, \bm{v}, s) \in \Sigma \times \bm{U} \times R$, the integration by parts formula \eqref{eq: integration by parts}, and the operator $\skw$, we obtain the following variational formulation of the problem \eqref{strong system fracture}: \\
Find $(\sigma, \bm{u}, r) \in \Sigma \times \bm{U} \times R$ such that 
\begin{subequations} \label{weakform}
	\begin{align}
		(\mathfrak{A} \sigma, \tau)_{\Omega \times \Gamma} 
		+ \inpo{\bm{u}}{\mathfrak{D} \cdot \epsilon \tau}
		+ \inpo{r}{\skw \epsilon \tau}
		&= (\bm{g}_u, \bm{n} \cdot \epsilon \tau)_{\partial_u \Omega},
		& \tau &\in \Sigma, \\
		\inpo{\mathfrak{D} \cdot \epsilon \sigma}{\bm{v}} &= \inpo{\epsilon^2 \bm{f}}{\bm{v}},
		& \bm{v} &\in \bm{U}, \\
		\inpo{\skw \epsilon \sigma}{s} &= 0, & s &\in R.
	\end{align}
\end{subequations}

We identify system \eqref{weakform} as a saddle point problem by introducing the bilinear forms $a: \Sigma \times \Sigma \to \mathbb{R}$ and $b: \Sigma \times (\bm{U} \times R) \to \mathbb{R}$:
\begin{subequations}  \label{operators}
	\begin{align}
		a(\sigma; \tau) &:=
			(\mathfrak{A} \sigma, \tau)_{\Omega \times \Gamma}   \label{operator a}\\
		b(\sigma; \bm{v}, s) &:= \inpo{\mathfrak{D} \cdot \epsilon \sigma}{\bm{v}}
		+ \inpo{\skw \epsilon \sigma}{s}. \label{operator b}
	\end{align}
\end{subequations}

The problem \eqref{weakform} can then be rewritten to the following, equivalent formulation: \\
Find $(\sigma, \bm{u}, r) \in \Sigma \times \bm{U} \times R$ such that 
\begin{subequations}
	\begin{align}
		a(\sigma; \tau)
		+ b(\tau; \bm{u}, r) 
		&= (\bm{g}_u, \bm{n} \cdot \epsilon \tau)_{\partial_u \Omega} \\
		b(\sigma; \bm{v}, s) 
		&= \inpo{\epsilon^2 \bm{f}}{\bm{v}}
	\end{align}
\end{subequations}
for all $(\tau, \bm{v}, s) \in \Sigma \times \bm{U} \times R$.

\section{Well-Posedness} \label{sec: Well-posedness}

In this section, we show well-posedness of the continuous formulation \eqref{weakform}. The key is to associate appropriately weighted norms to the function spaces introduced in the previous section. In the mixed-dimensional setting considered here, let us endow $\Sigma$, $\bm{U}$, and $R$ with the following norms
\begin{subequations} \label{norms}
	\begin{align}
		\| \tau \|_\Sigma &= (\normL{\tau}^2 + \normLg{ \bm{n} \cdot \tau}^2 + \normL{ \hat{\epsilon}_{\max}^{-1} \mathfrak{D} \cdot \epsilon \tau}^2)^{1/2}, \\
		\| \bm{v} \|_{U} &= \normL{\hat{\epsilon}_{\max} \bm{v}}, \\
		\| s \|_{R} &= \normL{\epsilon s}.
	\end{align}
\end{subequations}

The proof of well-posedness consists of proving sufficient conditions on the bilinear forms $a$ and $b$ from \eqref{operators} to invoke standard saddle-point theory.
First, we show continuity of the operators, followed by ellipticity of $a$ and inf-sup on $b$.
\begin{lemma}[Continuity] \label{lem: continuity}
	The bilinear forms $a$ and $b$ from \eqref{operators} are continuous with respect to the norms given by \eqref{norms}.
\end{lemma}
\begin{proof}
	The continuity of $a$ follows from \eqref{Bounds A}. 
	For the blinear form $b$, we derive
	\begin{align*}
		b(\sigma; \bm{v}, s) 
		&= \inpo{\mathfrak{D} \cdot \epsilon \sigma}{\bm{v}}
		+ \inpo{\skw \epsilon \sigma}{s} \nonumber\\
		&= \inpo{\hat{\epsilon}_{\max}^{-1} \mathfrak{D} \cdot \epsilon \sigma}{\hat{\epsilon}_{\max} \bm{v}}
		+ \inpo{\skw \sigma}{\epsilon s} \nonumber\\
		&\le \normL{\hat{\epsilon}_{\max}^{-1} \mathfrak{D} \cdot \epsilon \sigma} 
		\normL{\hat{\epsilon}_{\max} \bm{v}}
		+ \normL{ \sigma}
		\normL{\epsilon s} \nonumber\\
		&\lesssim \norm{\sigma}_\Sigma (\norm{\bm{v}}_{U} + \norm{s}_{R}).
	\end{align*}
\qed
\end{proof}

Next, we focus on the bilinear form $a$. For the purposes of our analysis, it suffices to show that $a$ is elliptic on a specific subspace of $\Sigma$ generated by $b$. This is formally considered in the following lemma.
\begin{theorem}[Ellipticity]\label{Thm: ellipticity}
	Given the bilinear forms $a$ and $b$ from \eqref{operators}. If $\sigma \in \Sigma$ satisfies
	\begin{align}
		b(\sigma; \bm{v}, s) &= 0, & \text{for all } ( \bm{v}, s) \in \bm{U} \times R, \label{theorem condition}
	\end{align}
	then the following ellipticity bound holds
	\begin{align*}
		a(\sigma; \sigma) \gtrsim \norm{\sigma}_\Sigma^2.
	\end{align*}
\end{theorem}
\begin{proof}
	We set $s = 0$ in condition \eqref{theorem condition}. The assumption holds for all $\bm{v} \in \bm{U}$, thus noting that $\mathfrak{D} \cdot \epsilon \sigma \in \prod_{i \in I} L^2(\Omega_i) = \bm{U}$ and $\hat{\epsilon}_{\max} > 0$, we obtain
	\begin{align}
		\normL{\hat{\epsilon}_{\max}^{-1} \mathfrak{D} \cdot \epsilon \sigma} = 0, \label{eq: b condition}
	\end{align}
	The proof is concluded by combining \eqref{eq: b condition} with the coercivity of $\mathfrak{A}$ from \eqref{Bounds A}.
\qed
\end{proof}

With the properties of $a$ proven, we continue by considering an inf-sup condition on the bilinear form $b$. 
This is shown in the following theorem, which relies on the constructions from Lemmas~\ref{lem: Poisson} and \ref{lem: Stokes}, presented afterwards.

\begin{theorem}[Inf-Sup]\label{Thm: inf-sup}
	The bilinear form $b$ satisfies for all, $( \bm{u}, r) \in \bm{U} \times R$,
	\begin{align*}
		\sup_{\tau \in \Sigma} \frac{b(\tau; \bm{u}, r)}{\norm{\tau}_\Sigma} 
		\gtrsim \norm{\bm{u}}_U + \norm{r}_R,
	\end{align*}
\end{theorem}
\begin{proof}
	The proof consists of constructing a suitable $\tau \in \Sigma$ for a given pair $( \bm{u}, r) \in \bm{U} \times R$. Its construction is based on constructing two auxiliary functions $\eta, \xi \in \Sigma$ using the techniques from Lemmas \ref{lem: Poisson} and \ref{lem: Stokes}. Setting $\tau$ as the sum of these two functions then yields the result.

	First, Lemma~\ref{lem: Poisson} allows us to construct $\eta \in \Sigma$ such that
	\begin{align}
		\mathfrak{D} \cdot \epsilon \eta &= \hat{\epsilon}_{\max}^2 \bm{u}, &
		\| \eta \|_{\Sigma} &\lesssim \| \bm{u} \|_U. \label{bound eta}
	\end{align}

	Secondly, we choose $\xi \in \Sigma$ using Lemma~\ref{lem: Stokes} with given $(r - \epsilon^{-1} \skw \eta) \in R$ such that
	\begin{subequations}
		\begin{align}
			\skw \xi
			&= \epsilon r - \skw \eta \label{constr wy}\\
			\mathfrak{D} \cdot \epsilon \xi
			&= 0 \\ 
			\norm{\xi}_\Sigma &\lesssim \norm{ r }_R + \norm{ \epsilon^{-1} \skw \eta}_R 
			= \norm{ r }_R + \norm{ \skw \eta}_{\Omega} 
			\le \norm{ r }_R + \norm{\eta}_\Sigma. \label{bound wy}
		\end{align}
	\end{subequations}
	
	By setting $\tau = \eta + \xi$, it follows that
	\begin{align}
		b(\tau; \bm{u}, r) 
		&= \inpo{\mathfrak{D} \cdot \epsilon \tau}{\bm{u}}
		+ \inpo{\skw \epsilon\tau}{r} \nonumber\\
		&= \inpo{\mathfrak{D} \cdot \epsilon \eta}{\bm{u}} + \inpo{\skw \eta +  \skw \xi}{ \epsilon r} \nonumber\\
		&= \normL{ \hat{\epsilon}_{\max} \bm{u} }^2 +  \norm{ \epsilon r }_{\Omega}^2 \nonumber\\
		&= \norm{ \bm{u} }_U^2 +  \norm{ r }_R^2 \label{bequals1}
	\end{align}
	Furthermore, the bound on $\tau$ is derived using \eqref{bound eta} and \eqref{bound wy}
	\begin{align}
		\| \tau \|_\Sigma
		\le \| \eta \|_\Sigma + \norm{ \xi}_\Sigma 
		\lesssim \norm{ \bm{u} }_U + \norm{ r }_R 
		. \label{bfinal}
	\end{align}
	The proof is concluded by combining \eqref{bequals1} and \eqref{bfinal}.
\qed
\end{proof}

\begin{lemma} \label{lem: Poisson}
	For each $\bm{u} \in \bm{U}$, a function $\eta \in \Sigma$ exists such that
	\begin{align}
			\mathfrak{D} \cdot \epsilon \eta &= \hat{\epsilon}_{\max}^2 \bm{u}, &
			\| \eta \|_{\Sigma} &\lesssim \| \bm{u} \|_U. \label{eq Poisson}
	\end{align}
\end{lemma}
\begin{proof}
	Considering $\bm{u} \in \bm{U}$ given, the function $\eta$ is constructed hierarchically. For each dimension $d$, we first set an interface value $\bm{\phi}$ on $\Gamma^d$, followed by a suitable extension into $\Omega^d$.

\begin{enumerate}\setcounter{enumi}{-1}
	\item
	Given $i \in I^0$, we construct the adjacent interface functions $\bm{\phi}_j \in L^2(\Gamma_j)$ such $\bm{\phi}_j = - \hat{\epsilon}_{\max} \bm{u}_i$ for a chosen $j \in \hat{J}_i$ with $(\hat{\epsilon})|_{\Gamma_j} = (\hat{\epsilon}_{\max})|_{\Omega_i}$ and zero for all other $j \in \hat{J}_i$. Repeating this construction for all $i \in I^0$, it follows that
	\begin{subequations}
	\begin{align}
			-\jump{\hat{\epsilon} \bm{\phi}}_i &= \hat{\epsilon}_{\max}^2 \bm{u}_i, & \forall i &\in I^0 \label{constr psi^0}\\
			\norm{\bm{\phi}}_{\Gamma^0} &= \norm{\hat{\epsilon}_{\max} \bm{u}}_{\Omega^0}. \label{bound psi^0}
		\end{align}\end{subequations}

	\item We continue with $i \in I^1$ and perform the following two steps. First, the function $\eta_i$ is constructed as the bounded $H(\operatorname{div}, \Omega_i)$-extension of the given $\bm{\phi}_j$ with $j \in \check{J}_i$. We use the extension operator as described in \cite{quarteroni1999domain} (Section 4.1.2), giving us the properties
	\begin{subequations}
		\begin{align}
			(\bm{n} \cdot \eta_i)|_{\Gamma_j} &= \bm{\phi}_j, & \forall j &\in \check{J}_i,\\
			(\bm{n} \cdot \eta_i)|_{\partial \Omega_i \setminus \Gamma^0} &= 0, \\
			\norm{\eta}_{\Omega_i} + \norm{\nabla \cdot \eta}_{\Omega_i} &\lesssim 
			\sum_{j \in \check{J}_i} \norm{\bm{\phi}}_{\Gamma_j}. \label{bound eta^d}
		\end{align}
	\end{subequations}
	Secondly, we further define $\bm{\phi}$ onto $\Gamma_j$ with $j \in \hat{J}_i$. We choose a single $j \in \hat{J}_i$ where $(\hat{\epsilon})|_{\Gamma_j} = (\hat{\epsilon}_{\max})|_{\Omega_i}$ and set
	$\bm{\phi}_j = - \hat{\epsilon}_{\max} \bm{u}_i + \hat{\epsilon}_{\max}^{-1} \nabla \cdot \epsilon \eta_i$. For all other $j \in \hat{J}_i$, we set $\bm{\phi}_j = 0$. It then immediately follows that
	\begin{align}
		-\jump{\hat{\epsilon} \bm{\phi}}_i 
		&= \hat{\epsilon}_{\max}^2 \bm{u}_i 
		- \nabla \cdot \epsilon \eta_i, \label{constr psi^1}
	\end{align}
	Repeating these two steps for all $i \in I^1$ gives us the bound	
	\begin{align}
		\norm{\bm{\phi}}_{\Gamma^1} 
		&\le \norm{\hat{\epsilon}_{\max} \bm{u}}_{\Omega^1} + \norm{ \hat{\epsilon}_{\max}^{-1} \nabla \cdot \epsilon \eta}_{\Omega^1} \nonumber \\
		&\lesssim \norm{\hat{\epsilon}_{\max} \bm{u}}_{\Omega^1} 
		+ \norm{ \nabla \cdot \eta}_{\Omega^1}  \nonumber \\
		&\lesssim \norm{\hat{\epsilon}_{\max} \bm{u}}_{\Omega^1} 
		+ \norm{\bm{\phi}}_{\Gamma^0} \label{bound psi^1}
	\end{align}
	in which the second and third inequalities follow from \eqref{eps bounded by emax} and \eqref{bound eta^d}.

	\item 
	For $n = 3$, repeat the previous step for all $i \in I^2$ to obtain $\eta_i$ and $\bm{\phi}_j$ with $j \in \hat{J}_i$.

	\item 
	The construction of $\eta$ is finalized with its top-dimensional components $\eta_i$ with $i \in I^n$. Let the pair $(\eta_i, \tilde{\bm{v}}_i) \in (H(\div, \Omega_i))^n \times (L^2(\Omega_i))^n$ be the weak solution to the Poisson problem:
		\begin{subequations} \label{aux problem eta_i}
			\begin{align}
				\eta_i + \nabla \tilde{\bm{v}}_i &= 0,  \\
				\nabla \cdot \eta_i &= \bm{u}_i, & \label{constr eta_i} \\
				(\bm{n} \cdot \eta_i)|_{\Gamma_j} &= \bm{\phi}_j, & j \in \check{J}_i,\\
				(\bm{n} \cdot \eta_i)|_{\partial_\sigma \Omega_i} &= 0, \\
				(\tilde{\bm{v}}_i)|_{\partial_u \Omega_i} &= 0.
			\end{align}
		\end{subequations}	
	This problem is solved for all $i \in I^n$. We then recall that $\epsilon = \hat{\epsilon}_{\max} = 1$ in $\Omega^n$ and exploit the elliptic regularity of \eqref{aux problem eta_i} (see e.g. \cite{evans1998partial}) to obtain
	\begin{align}
		\norm{\eta}_{\Omega^n} 
		+ \norm{ \hat{\epsilon}_{\max}^{-1} \nabla \cdot \epsilon \eta}_{\Omega^n} 
		= \norm{\eta}_{\Omega^n} 
		+ \norm{ \nabla \cdot \eta}_{\Omega^n} 
		\lesssim
		\norm{\bm{\phi}}_{\Gamma^{n - 1}} 
		+ \norm{\bm{u}}_{\Omega^n}. \label{bound eta_i}
	\end{align}
\end{enumerate}

	With $\eta$ constructed, we consider its two main properties. First, by \eqref{constr psi^0}, \eqref{constr psi^1}, and \eqref{constr eta_i}, we have
	\begin{align}
		\mathfrak{D} \cdot \epsilon \eta
		&= \hat{\epsilon}_{\max}^2 \bm{u}. \label{constr eta}
	\end{align}
	Secondly, we find the following bound from \eqref{bound psi^0}, \eqref{bound eta^d}, \eqref{bound psi^1}, and \eqref{bound eta_i}:
	\begin{align}
		\| \eta \|_\Sigma^2 &= 
		\normL{\eta}^2 
		+ \norm{\bm{n} \cdot \eta}_{\Gamma}^2 
		+ \norm{\hat{\epsilon}_{\max}^{-1} \mathfrak{D} \cdot \epsilon \eta}_{\Omega}^2 
		\nonumber\\
		&\lesssim 
		\norm{\bm{\phi}}_{\Gamma}^2
		+ \norm{\hat{\epsilon}_{\max} \bm{u}}_{\Omega}^2 
		\lesssim \normL{\hat{\epsilon}_{\max} \bm{u}}^2
		= \norm{\bm{u}}_U^2, \label{bound eta psi}
	\end{align}
	thereby concluding the proof.
\qed
\end{proof}

Before introducing the second ingredient used in the proof of Theorem~\ref{Thm: inf-sup}, we require several key analytical tools, organized in the following diagram:
	\begin{equation} \label{eq: commuting_diagram}
		\begin{tikzcd}
			W \rar{\epsilon^{-1} \mathfrak{D} \times} \dar{\Xi} 
			&\Sigma \rar{\mathfrak{D} \cdot \epsilon} \dar{\skw \epsilon}
			&\bm{U} \\
			\Theta \rar{\widetilde{\mathfrak{D}} \cdot} & R
		\end{tikzcd}
	\end{equation}

The function spaces ($\Theta$ and $W$) and mappings ($\Xi$, $\widetilde{\mathfrak{D}} \cdot$, and $\mathfrak{D} \times$) are defined next. Let the auxiliary space $\Theta$ be given by
	\begin{align}
		\Theta 
		= \prod_{d = 2}^3 \prod_{i \in I^d} (H^1(\Omega_i))^{k_d \times d}.
	\end{align}
We emphasize that for an element $\theta \in \Theta$, this definition implies that $\theta_i$ is a 2-vector for $i \in I^2$ and a $3 \times 3$ tensor for $i \in I^3$.
Next, we follow \cite{awanou2013rectangular} by introducing the mapping $\Xi$ and its right-inverse $\Xi^{-1}$ as:
\begin{align*}
	(\Xi w)|_{\Omega_i} &= \left\{
	\begin{aligned}
		&(w_i)^T - \Tr(w_i)I, &i &\in I^3, \\
		&(w_{i, \|})^T, &i &\in I^2,
	\end{aligned} \right.
		\\
	(\Xi^{-1} \theta)|_{\Omega_i} &= \left\{
	\begin{aligned}
		&(\theta_i)^T - {\textstyle \frac{1}{2}} \Tr(\theta_i)I, &i &\in I^3, \\
		&[\theta_i,\ 0]^T, &i &\in I^2,
	\end{aligned} \right.
\end{align*}
with $w_{i, \|}$ the tangential components of $w_i$ with respect to $\Omega_i$. 
$W$ is defined as the space of functions that lie in the image of the inverse operator and have a mixed-dimensional curl in $\Sigma$, i.e.
\begin{align}
	W := \{ w \in \Xi^{-1} \Theta: \ \epsilon^{-1} \mathfrak{D} \times w \in \Sigma\}. 
\end{align}
We remark that for $w \in W$, $w_i$ is a 3-vector for $i \in I^3$ and a $3 \times 3$ tensor for $i \in I^3$.

Next, we introduce a divergence-like operator
$\widetilde{\mathfrak{D}} \cdot: \Theta \to R$ given by
\begin{align}
	\widetilde{\mathfrak{D}} \cdot \theta &= \left\{
	\begin{aligned}		
			\nabla \cdot \theta_i&, &i &\in I^3, \\
			\nabla \cdot \theta_i& - \check{\bm{n}}_i \cdot \jump{\bm{n} \cdot \hat{\theta}}_i, &i &\in I^2.
	\end{aligned}
	\right.
\end{align} 
Here, $\check{\bm{n}}_i$ is the unique unit vector normal to $\Omega_i$ that forms a positive orientation with the chosen basis of the tangential bundle. By definition, this divergence operator maps from $\Theta$ to $R$, and we emphasize that $\widetilde{\mathfrak{D}} \cdot \theta$ is a vector for $d = 3$ and a scalar for $d = 2$. 

Finally, the mixed-dimensional curl $(\mathfrak{D} \times )$ of $w \in W$ (see e.g. \cite{Licht,boon2017excalc}) is given by 
\begin{align}
	\mathfrak{D} \times w = \left\{ 
	\begin{aligned}		
			\nabla \times w_i&, & \forall i &\in I^3, \\
			\nabla^{\perp} w_i - \jump{\bm{n} \times \hat{w}}_i &, & \forall i &\in I^2, \\
			- \jump{\bm{n}^{\perp} \hat{w}}_i&, & \forall i &\in I^1.
	\end{aligned}
	\right.
\end{align}
Here, the superscript $\perp$ implies $[v_1, v_2]^\perp = [-v_2, v_1]$, e.g. $\nabla^\perp$ is a rotated gradient operator. We note that all differential operations are performed row-wise. Hence, for $n = 3$, the mixed-dimensional curl maps to a $3 \times 3$ tensor in $\Omega^3$, a $3 \times 2$ tensor in $\Omega^2$ (in local coordinates) and a 3-vector in $\Omega^1$ (in local coordinates). Thus, an exact correspondence with the function space $\Sigma$ is obtained, as reflected in the diagram.

\begin{lemma} \label{lem: commutativity}
	The operators in diagram \eqref{eq: commuting_diagram} enjoy the following two properties for all $w \in W$:
	\begin{align}
		\mathfrak{D} \cdot \mathfrak{D} \times w &= 0, &
		\skw \mathfrak{D} \times w &= \widetilde{\mathfrak{D}} \cdot \Xi w.
	\end{align}
\end{lemma}
\begin{proof}
	The top row of \eqref{eq: commuting_diagram} uses the differential operators from the mixed-dimensional De Rham complex \cite{boon2017excalc}. The first equality then follows from the fact that exact forms are closed. It remains to show commutativity. By the definition of $\Xi$, see e.g. \cite{awanou2013rectangular,Boffi}, we have 
	\begin{align*}
		\skw_3 \nabla \times w_i &= \nabla \cdot \Xi_3 w_i, & \forall i &\in I^3, \\
		\skw_2 \nabla^\perp w_i &= \nabla \cdot \Xi_2 w_i, & \forall i &\in I^2,
	\end{align*} 
	in which the subscript $d$ denotes a restriction of the operator to $\Omega^d$. Furthermore, we note that on $\Omega_i$ with $i \in I^2$, the skw operator evaluates the asymmetry with respect to the tangent bundle of $\Omega_i$. In turn, we have for $M \in \mathbb{R}^{3 \times 3}$ that $(\skw_2 M)|_{\Omega_i} = \check{\bm{n}}_i \cdot \skw_3 M$. This gives us
		\begin{align*} 
			(\skw \mathfrak{D} \times w)|_{\Omega_i}
			&= 	
			\skw_3 \nabla \times w_i
			= \nabla \cdot \Xi_3 w_i
			= (\widetilde{\mathfrak{D}} \cdot \Xi w)|_{\Omega_i}
			, & \forall i &\in I^3,\\
			(\skw \mathfrak{D} \times w)|_{\Omega_i}
			&= \skw_2 (\nabla^{\perp} w_i - \jump{\bm{n} \times \hat{w}}_i) \\
			&= \nabla \cdot \Xi_2 w_i - \check{\bm{n}}_i \cdot \jump{\bm{n} \cdot \Xi_3 \hat{w}}_i \\
			&= (\widetilde{\mathfrak{D}} \cdot \Xi w)|_{\Omega_i}, & \forall i &\in I^2.
		\end{align*}
\qed
\end{proof}

\begin{lemma} \label{lem: Stokes}
	Given $r \in R$, a function $\xi \in \Sigma$ exists such that
	\begin{align}
			\mathfrak{D} \cdot \epsilon \xi &= 0, &
			\skw \xi &= \epsilon r, &
			\| \xi \|_{\Sigma} &\lesssim \| r \|_R. \label{eq Stokes}
	\end{align}
\end{lemma}
\begin{proof}
	We give the proof for $n = 3$, the case $n = 2$ being simpler. 	
	The strategy is to exploit the properties shown in Lemma~\ref{lem: commutativity} and first construct a bounded $\theta \in \Theta$ such that $\widetilde{\mathfrak{D}} \cdot \theta = \epsilon^2 r$. Then, by setting $w = \Xi^{-1} \theta$ and $\xi = \epsilon^{-1} \mathfrak{D} \times w$, we obtain two of the desired properties
	\begin{align} \label{eq: desired}
		\mathfrak{D} \cdot \epsilon \xi &= 
		\mathfrak{D} \cdot \mathfrak{D} \times w = 
		0, &
		\skw \xi &= 
		\epsilon^{-1} \skw \mathfrak{D} \times w = 
		\epsilon^{-1} \widetilde{\mathfrak{D}} \cdot \theta = 
		\epsilon r.
	\end{align}
	The estimate will then follow from the boundedness of $\theta$.

	The construction of $\theta$ proceeds according to the following three steps, consisting of an interface function $\phi \in H^1(\Gamma^2)$ which serves as a source function for $\theta_i$ with $i \in I^2$ and a boundary condition for $\theta_i$ with $i \in I^3$.

	\begin{enumerate}
		\item
		We start by defining a scalar function $\phi$ in the trace space $H^1(\Gamma^2)$. We let $\phi$ vanish at all intersections and extremities, i.e. $\phi$ is in the function space $\Phi$ given by
		\begin{align}
			\Phi &:= \prod_{i \in I^2} \prod_{j \in \hat{J}_i} H_0^1(\Gamma_j).
		\end{align}
		Now, let $\phi$ be the solution to the following minimization problem:
		\begin{align}
				\min_{\varphi \in \Phi} &\ \tfrac{1}{2} \norm{\varphi}_{H^1(\Gamma^2)}^2 &
			 	& \text{subject to } \ 
				\Pi_{\mathbb{R}_i} (\jump{\varphi}_i + \epsilon^2 r_i) = 0, \ \forall i \in I^2.
			\end{align}
		with $\Pi_{\mathbb{R}_i}$ the projection onto constants on $\Omega_i$. Due to the regularity of this problem and the imposed constraint, we have
		\begin{subequations}
			\begin{align}
				- \Pi_{\mathbb{R}_i} \jump{\phi}_i &= \Pi_{\mathbb{R}_i} \epsilon^2 r_i,  & \forall i &\in I^2 \label{constraints V}\\
				\| \phi \|_{H^1(\Gamma^2)} &\lesssim \| \epsilon^2 r \|_{\Omega^2}. \label{Stokes bound theta}
			\end{align}
		\end{subequations}

		\item
		For each $i \in I^2$, we construct a function $\theta_i$ using $\phi$ as a source function. Specifically, let $(\theta_i, p_i) \in (H_0^1(\Omega_i))^2 \times L^2(\Omega_i)$ be the weak solution to the Stokes problem:
		\begin{subequations}
		\begin{align} \label{eq: Stokes n-1}
				\nabla \cdot (\nabla \theta_i) - \nabla p_i &= 0 \\
				\nabla \cdot \theta_i &= (I - \Pi_{\mathbb{R}_i})(\epsilon^2 r_i + \jump{\phi}_i), \label{div r n-1}\\
				(\theta_i)|_{\partial \Omega_i} &= 0.\
		\end{align}
		\end{subequations}

		The following bound is then satisfied from the regularity of \eqref{eq: Stokes n-1} (see e.g. \cite{evans1998partial}) combined with \eqref{Stokes bound theta}
		\begin{align}
			\norm{\theta}_{H^1(\Omega^2)} &\lesssim \norm{\epsilon^2 r}_{\Omega^2} + \norm{\phi}_{\Gamma^2}
			\lesssim \norm{\epsilon^2 r}_{\Omega^2} \label{Stokes bound wn-1}
		\end{align}

		\item
		To finalize $\theta \in \Theta$, we create $\theta_i$ for $i \in I^3$ using $\phi$ from the first step as a boundary condition. Let $\theta_i \in (H^1(\Omega_i))^{3 \times 3}$ and an auxiliary pressure variable $p_i \in (L^2(\Omega_i))^3$ be the weak solution to the following Stokes problem:
		\begin{subequations} \label{eq: Stokes}
			\begin{align}
				\nabla \cdot (\nabla \theta_i) - \nabla p_i &= 0, \\
				\nabla \cdot \theta_i &= r_i,  \label{div r n}\\
				(\theta_i)|_{\Gamma_j} &= \phi_j \check{\bm{n}}_i \bm{n}_j^T, & j &\in \check{J}_i, \\
				(p_i)|_{\partial \Omega_i \setminus \Gamma^2} &=0.
			\end{align}
		\end{subequations}

		Recall that $\check{\bm{n}}_i$, the unique normal vector of $\Omega_i$, and $\bm{n}_j$, the normal vector defined on $\Gamma_j$, are equal up to sign. This problem is well-posed since $\partial \Omega_i \setminus \Gamma^2$ has positive measure, for each $i \in I^3$, by assumption. 
		We have the following bound due to the regularity of the Stokes problems and the fact that $\epsilon = 1$ in $\Omega^3$
		\begin{align}
			\norm{\theta}_{H^1(\Omega^3)} 
			\lesssim \norm{r}_{\Omega^3} + \norm{\phi}_{\Gamma^2}
			\lesssim \norm{\epsilon^2 r}_{\Omega^3} + \norm{\epsilon^2 r}_{\Omega^2}. \label{Stokes bound wn}
		\end{align}
		\end{enumerate}

		Combining all $\theta_i$ from the final two steps gives us $\theta \in \Theta$. We first note that the properties \eqref{constraints V}, \eqref{div r n-1}, and \eqref{div r n} result in
		\begin{align*}
			(\widetilde{\mathfrak{D}} \cdot \theta)|_{\Omega_i}
			&= 
			\nabla \cdot \theta_i - \check{\bm{n}}_i \cdot \jump{\bm{n} \cdot \theta_i} \nonumber \\
			&= (I - \Pi_{\mathbb{R}_i})(\epsilon^2 r_i + \jump{\phi}_i) - \jump{\phi}_i
			= \epsilon^2 r_i
			= (\epsilon^2 r)|_{\Omega_i}, 
			& \forall i &\in I^2, \\
			(\widetilde{\mathfrak{D}} \cdot \theta)|_{\Omega_i}
			&= 
			\nabla \cdot \theta_i 
			= r_i 
			= (\epsilon^2 r)|_{\Omega_i},
			& \forall i &\in I^3. 
		\end{align*}
		Hence, we have $\widetilde{\mathfrak{D}} \cdot \theta = \epsilon^2 r$ and it follows from \eqref{eq: desired} that setting $w = \Xi^{-1} \theta$ and $\xi = \epsilon^{-1} \mathfrak{D} \times w$ provides the first two properties. The bound follows due to \eqref{Stokes bound theta}, \eqref{Stokes bound wn-1} and \eqref{Stokes bound wn}
		\begin{align}
			\norm{\xi}_\Sigma^2
			&= \normL{\epsilon^{-1} \mathfrak{D} \times w}^2 + \normLg{\bm{n} \cdot (\epsilon^{-1} \mathfrak{D} \times w)}^2 \nonumber\\
			&\lesssim \| \epsilon^{-1} w \|_{H^1(\Omega)}^2 
				+ \| \hat{\epsilon}^{-1} \hat{w} \|_{H^1(\Gamma)}^2 \nonumber\\
			&\lesssim \| \epsilon^{-1} \theta \|_{H^1(\Omega)}^2 
				+ \| \hat{\theta} \|_{H^1(\Gamma^2)}^2 \nonumber\\
			&= \| \epsilon^{-1} \theta \|_{H^1(\Omega)}^2 
				+ \| \phi \|_{H^1(\Gamma^2)}^2 \nonumber\\
			&\lesssim \norm{r}_R^2, \label{eq: stokes2}
		\end{align}
		as desired.
\qed
\end{proof}

With the proven properties of the bilinear forms $a$ and $b$, the main result of this section is summarized by the following theorem:
\begin{theorem}
	Problem \eqref{weakform} is well-posed with respect to the norms \eqref{norms}. That is, a unique solution exists satisfying the bound
	\begin{align}
		\norm{\sigma}_\Sigma + \norm{\bm{u}}_U + \norm{r}_R \lesssim \norm{\hat{\epsilon}_{\max} \bm{g}_u}_{H^{\frac{1}{2}}(\partial_u \Omega)} + \normL{\epsilon \bm{f}}
	\end{align}
\end{theorem}
\begin{proof}
	It suffices to show continuity of the right-hand side of \eqref{weakform} with respect to the norms above. For that purpose, we derive the following bound on the first term using Cauchy-Schwarz and a trace inequality:
	\begin{align}
		(\bm{g}_u, \bm{n} \cdot \epsilon \tau)_{\partial_u \Omega} 
		&\lesssim \norm{\hat{\epsilon}_{\max} \bm{g}_u}_{H^{\frac{1}{2}}(\partial_u \Omega)} \norm{\hat{\epsilon}_{\max}^{-1} \epsilon \tau}_{H(\div, \Omega)} \nonumber\\
		&\lesssim \norm{\hat{\epsilon}_{\max} \bm{g}_u}_{H^{\frac{1}{2}}(\partial_u \Omega)} (\normL{\tau} + \normL{ \hat{\epsilon}_{\max}^{-1} \mathfrak{D} \cdot \epsilon \tau} + \normL{ \hat{\epsilon}_{\max}^{-1} \jump{\bm{n} \cdot \hat{\epsilon} \tau}}) \nonumber\\
		&\lesssim \norm{\hat{\epsilon}_{\max} \bm{g}_u}_{H^{\frac{1}{2}}(\partial_u \Omega)} \norm{\tau}_\Sigma.
	\end{align}
	Moreover, from \eqref{eps bounded by emax}, the second term is bounded as follows
	\begin{align}
		\inpo{\epsilon^2 \bm{f}}{\bm{v}} 
		&\le \normL{\hat{\epsilon}_{\max}^{-1} \epsilon^2 \bm{f}} \normL{\hat{\epsilon}_{\max} \bm{v}}
		\lesssim \normL{\epsilon \bm{f}} \norm{\bm{v}}_U
	\end{align}
	The result now follows readily from these two estimates, Theorems \ref{Thm: ellipticity} and \ref{Thm: inf-sup}, and standard saddle point theory \cite{Boffi}.
\qed
\end{proof}

\section{Discretization} \label{sec: Discretization}

In this section, we discretize the continuous problem \eqref{weakform} using conforming finite elements. The notion of conformity and the choice of mixed finite element spaces is discussed in Section~\ref{sub: Discrete Spaces} and the resulting discretized problem is presented and analyzed in Section~\ref{sub: Discrete Problem}.

\subsection{Discrete Spaces} \label{sub: Discrete Spaces}

For each $i \in I$, we introduce a shape-regular, simplicial grid $\Omega_{i, h}$ which tessellates $\Omega_i$. The union of meshes of a given dimension $d$ (with $0 \le d \le n$) is denoted by $\Omega_h^d = \cup_{i \in I^d} \Omega_{i, h}$ and we define $\Omega_h = \cup_{i \in I} \Omega_{i, h}$. We let the grid respect all lower-dimensional features and be matching across all interfaces. The tesselation of $\Gamma$ is thus given by $\Gamma_h = \Gamma \cap \partial \Omega_h$. Moreover, there is an equivalence between $\Gamma_{j, h}$ and $\Omega_{i, h}$ for all $j \in \hat{J}_i$. The typical mesh size is denoted by $h$ and we use $h$ as a subscript to indicate the discretized counterpart of functions and function spaces.

With the aim of obtaining a stable and conforming method, we search for a discrete solution in subspaces of the function spaces defined in Section~\ref{sec:weak_form}. We choose discrete function spaces $(\Sigma_h, \bm{U}_h, R_h)$ on the grid $\Omega_h$ according to the following three conditions:
\begin{enumerate}[label = (S{\arabic*})] \label{S}
	\item The finite element spaces are conforming, i.e. \label{S: conformity}
	\begin{align*}
		\Sigma_h &\subset \Sigma, & \bm{U}_h &\subset \bm{U}, & R_h &\subset R.
	\end{align*}
	\item $\Sigma_h$ and $\bm{U}_h$ are such that for each $i \in I$: \label{S: Poisson}
	\begin{align*}
		\nabla \cdot \Sigma_{i, h} 
		&\subseteq \bm{U}_{i, h} &
		&\text{and} &
		(\bm{n} \cdot \Sigma_{\hatj, h})|_{\Gamma_j}
		&= \bm{U}_{i, h}, &
		\forall j \in \hat{J}_i.
	\end{align*}
	\item A mixed-dimensional, finite element space $W_h$ exists such that 
	\begin{enumerate}
		\item $\mathfrak{D} \times W_h \subseteq \Sigma_h$.
		\item $(\Xi W_{i, h}) \times R_{i, h}$ forms a stable pair for the two-dimensional Stokes problem for each $i \in I^2$.
	\end{enumerate}
	\label{S: W_h}
	\item $\Sigma_{i, h} \times \bm{U}_{i, h} \times R_{i, h}$ forms a stable triplet for three-dimensional, mixed elasticity for each $i \in I^3$. \label{S: stable triplet}
\end{enumerate}

We provide an exemplary family of finite elements satisfying all four conditions. This choice is most concisely described using the notation of finite element exterior calculus \cite{AFW_FEEC}. A translation to more conventional nomenclature is provided afterwards, for convenience. Given a polynomial degree $k \ge 0$, let
\begin{subequations}\label{eq: Finite Elements}
	\begin{align} 
		\Sigma_h &= \prod_{d = 1}^n \prod_{i \in I^d} 
			\left( P_{k + n - d + 1} \Lambda^{d - 1}(\Omega_{i, h}) \right)^n, \\
		\bm{U}_h &= \prod_{d = 0}^n \prod_{i \in I^d} 
			\left( P_{k + n - d} \Lambda^d(\Omega_{i, h}) \right)^n, \\
		R_h &= \prod_{d = 2}^n \prod_{i \in I^d} 
			\left( P_{k + n - d} \Lambda^d(\Omega_{i, h}) \right)^{k_d}.
	\end{align}

In other words, for $n = 3$: $\Sigma_{i, h}$ with $i \in I^3$ corresponds to three rows of Nedelec elements of the second kind ($N2^f_{k + 1}$ \cite{Nedelec}) with degrees of freedom on the faces. For $i \in I^2$, it is three rows of Brezzi-Douglas-Marini elements ($BDM_{k + n - 1}$ \cite{brezzi1985two}). Finally $\Sigma_{i, h}$ with $i \in I^1$ is given by a triplet of continuous Lagrange elements ($P_{k + 2}$). The spaces $\bm{U}_{i, h}$ and $R_{i, h}$ are defined for $i \in I$ as three and $k_{d_i}$ rows, respectively, of discontinuous Lagrange elements ($P_{- (k + n - d_i)}$). 

In this case, the auxiliary space $W_h$ of \ref{S: W_h} is explicitly given by
	\begin{align} 
		W_h &= \prod_{d = 2}^n \prod_{i \in I^d} 
			\left( P_{k + n - d + 2}^- \Lambda^{d - 2}(\Omega_{i, h}) \right)^n.
	\end{align}
\end{subequations}
For $n = 3$, $W_{i, h}$ is thus given by three rows of (first kind) edge-based Nedelec element ($N1_{k + 2}^e$) for $i \in I^3$ and by three instances of Lagrange elements ($P_{k + n}$) for $i \in I^2$. 
The lowest order choice in this family, i.e. with $k= 0$, is presented in Table~\ref{tab: W_h}.

A reduced family of finite elements arises by noting that all stability conditions remain valid after the polynomial order of the trace onto $\Gamma$ is reduced by one. Table~\ref{tab: W_h reduced} presents the lowest order member of this family.

\begin{table}[th]
	\caption{The finite element spaces chosen for $n = 2$ and $n = 3$ of lowest order within the family \eqref{eq: Finite Elements}. The negative orders in the subscript denote discontinuous Lagrange elements. On the zero-dimensional manifolds, the polynomial order is redundant since any finite element space corresponds to point evaluation there.}
	\label{tab: W_h}
	\centering
	\begin{tabular}{l|rrrr}
	\hline
	  $d$  & $W_h$         & $\Sigma_h$  & $\bm{U}_h$   & $R_h$ \\ 
	\hline
		 2 & $(P_2)^2$     & $(BDM_1)^2$ & $(P_0)^2$    & \rule{0pt}{2.6ex}$P_0$ \\
		 1 &               & $(P_2)^2$   & $(P_{-1})^2$ &       \\                
		 0 &               &             & $(P_{-2})^2$ &       \\                
	\hline
	\end{tabular}
	\\[5pt]
	\begin{tabular}{l|rrrr}
	\hline
	   $d$ & $W_h$        & $\Sigma_h$   & $\bm{U}_h$   & $R_h$     \\ 
	\hline 
		 3 & $(N1_2^e)^3$ & $(N2_1^f)^3$ & $(P_0)^3$    & \rule{0pt}{2.6ex}$(P_0)^3$ \\
		 2 & $(P_3)^3$    & $(BDM_2)^3$  & $(P_{-1})^3$ & $P_{-1}$  \\ 
		 1 &              & $(P_3)^3$    & $(P_{-2})^3$ &           \\
		 0 &              &              & $(P_{-3})^3$ &           \\
	\hline
	\end{tabular}
\end{table}

\begin{table}[th]
	\caption{The lowest-order finite element spaces of the reduced family for $n = 2$ and $n = 3$. The superscript minus indicates that the trace of the finite element space onto $\Gamma_h$ is reduced by one order.}
	\label{tab: W_h reduced}
	\centering
	\begin{tabular}{l|rrrr}
	\hline
	  $d$  & $W_h$ & $\Sigma_h$  & $\bm{U}_h$   & $R_h$ \\
	\hline
		 2 & $(P_2^-)^2$ & $(BDM_1^-)^2$ & $(P_0)^2$    & \rule{0pt}{2.6ex}$P_0$ \\
		 1 & & $(P_1)^2$   & $(P_0)^2$ &       \\                
		 0 & &             & $(P_0)^2$ &       \\                
	\hline
	\end{tabular}
	\\[5pt]
	\begin{tabular}{l|rrrr}
	\hline
	   $d$ & $W_h$ & $\Sigma_h$   & $\bm{U}_h$   & $R_h$     \\
	\hline 
		 3 & $(N1_2^{e-})^3$	& $(N2_1^-)^3$ & $(P_0)^3$    & \rule{0pt}{2.6ex}$(P_0)^3$\\
		 2 & $(P_2^-)^3$ 	& $(BDM_1^-)^3$  & $(P_0)^3$    & $P_0$  \\ 
		 1 & 				& $(P_1)^3$    & $(P_0)^3$ &           \\
		 0 & 				&            & $(P_0)^3$ &           \\
	\hline
	\end{tabular}
\end{table}

Properties \ref{S: conformity}--\ref{S: W_h} can be verified with the use of the presented tables. Finally, these families correspond to the stable triplets analyzed in \cite{AFW_FEEC} (Sections 11.6--11.7), hence \ref{S: stable triplet} holds as well.

\subsection{Discrete Problem} \label{sub: Discrete Problem}

Since the finite elements described above are contained in the continuous spaces from Section~\ref{sec: Well-posedness} by \ref{S: conformity}, the discrete formulation of the model problem is a direct translation of \eqref{weakform}: \\
Find $(\sigma_h, \bm{u}_h, r_h) \in \Sigma_h \times \bm{U}_h \times R_h$ such that
\begin{subequations} \label{weakform_h}
	\begin{align}
		(\mathfrak{A} \sigma_h, \tau_h)_{\Omega \times \Gamma}
		+ \inpo{\bm{u}_h}{\mathfrak{D} \cdot \epsilon \tau_h}
		+ \inpo{r_h}{\skw \epsilon \tau_h}
		&= (\bm{g}_u, \bm{n} \cdot \epsilon \tau_h)_{\partial_u \Omega}, 
		\\
		\inpo{\mathfrak{D} \cdot \epsilon \sigma_h}{\bm{v}_h} &= \inpo{\epsilon^2 \bm{f}}{\bm{v}_h}, 
		\\
		\inpo{\skw \epsilon \sigma_h}{s_h} &= 0.
	\end{align}
\end{subequations}
for all $(\tau_h, \bm{v}_h, s_h ) \in \Sigma_h \times \bm{U}_h \times R_h$.
We note that the saddle-point structure of this problem has not changed, and can readily be uncovered using the bilinear forms from \eqref{operators}. 

\subsection{Stability} \label{subsec: well-posedness_h}

We continue with the analysis concerning the well-posedness of \eqref{weakform_h}. Let us recall the norms from \eqref{norms} for convenience
\begin{subequations} \label{norms_h}
	\begin{align}
		\| \tau \|_\Sigma^2 &= \normL{\tau}^2 + \normLg{\bm{n} \cdot \tau}^2 + \normL{ \hat{\epsilon}_{\max}^{-1} \mathfrak{D} \cdot \epsilon \tau}^2, \\
		\| \bm{v} \|_{U}^2 &= \normL{\hat{\epsilon}_{\max} \bm{v}}^2, \\
		\| s \|_{R}^2 &= \normL{\epsilon s}^2.
	\end{align}
\end{subequations}

\begin{theorem}[Ellipticity] \label{thm: ellip_h}
	If the discrete spaces satisfy \ref{S: Poisson} and $\sigma_h \in \Sigma_h$ satisfies
	\begin{align*}
		b(\sigma_h; \bm{v}_h, s_h) &= 0, & \text{for all } ( \bm{v}_h, s_h) \in \bm{U}_h \times R_h,
	\end{align*}
	then the following ellipticity bound holds
	\begin{align*}
		a(\sigma_h; \sigma_h) \gtrsim \norm{\sigma_h}_\Sigma^2.
	\end{align*}
\end{theorem}
\begin{proof}
	By \ref{S: Poisson}, we have $\mathfrak{D} \cdot \Sigma_h \subseteq \bm{U}_h$ and the same arguments are used as in Theorem~\ref{Thm: ellipticity}.
\qed
\end{proof}

The next step is to consider the inf-sup condition for the bilinear form $b$ in the discrete case. 

\begin{theorem}[Inf-Sup]\label{Thm: inf-sup_h}
	If the discrete spaces satisfy conditions \ref{S: conformity}-\ref{S: stable triplet}, then for all $(\bm{u}_h, r_h) \in \bm{U}_h \times R_h$, 
	\begin{align*}
		\sup_{\tau_h \in \Sigma_h} \frac{b(\tau_h; \bm{u}_h, r_h)}{\norm{\tau_h}} \gtrsim \norm{\bm{u}_h}_U + \norm{r_h}_R,
	\end{align*}
\end{theorem}
\begin{proof}
	With $\bm{u}_h \in \bm{U}_h$ and $r_h \in R_h$ given, we follow a similar strategy as in the proof of Theorem~\ref{Thm: inf-sup}. Here, we rely on Lemmas \ref{lem: Poisson h} and \ref{lem: Stokes h} to provide $\eta_h \in \Sigma_h$ to gain control of $\bm{u}_h$ and a divergence-free function $\xi_h \in \Sigma_h$ controlling $r_h \in R_h$.

	In short, we choose $\eta_h, \xi_h \in \Sigma_h$ such that
	\begin{align*}
		\mathfrak{D} \cdot \epsilon \eta_h &= \hat{\epsilon}_{\max}^2 \bm{u}_h, \\
		\mathfrak{D} \cdot \epsilon \xi_h &= 0, &
		\Pi_{R_h} \skw \xi_h &= \epsilon r_h - \Pi_{R_h} \skw \eta_h
	\end{align*}
	and satisfy the bound
	\begin{align*}
		\| \eta_h \|_{\Sigma} + \| \xi_h \|_{\Sigma} &\lesssim 
		\| \bm{u}_h \|_U + \| r_h \|_R
	\end{align*}
	
	Following the same steps as in Theorem~\ref{Thm: inf-sup}, we define $\tau_h = \eta_h + \xi_h$ so that
	\begin{align*}
		b(\tau_h; \bm{u}_h, r_h) 
		&= \inpo{\mathfrak{D} \cdot \epsilon \tau_h}{\bm{u}_h}
		+ \inpo{\skw \epsilon\tau_h}{r_h} \nonumber \\
		&= \inpo{\mathfrak{D} \cdot \epsilon \eta_h}{\bm{u}_h} + 
		\inpo{\skw \eta_h + \skw \xi_h}{\epsilon r_h} \nonumber \\
		&= \| \bm{u}_h \|_{U}^2  + \norm{ r_h }_R^2 \\
		\| \tau_h \|_\Sigma
		&\lesssim \| \eta_h \|_\Sigma + \| \xi_h \|_\Sigma
		\lesssim \| \bm{u}_h \|_{U}  + \norm{ r_h }_R.
	\end{align*}
	The proof is concluded by combining the above.
\qed
\end{proof}

\begin{lemma} \label{lem: Poisson h}
	For each $\bm{u}_h \in \bm{U}_h$, a function $\eta_h \in \Sigma_h$ exists such that
	\begin{align}
			\mathfrak{D} \cdot \epsilon \eta_h &= \hat{\epsilon}_{\max}^2 \bm{u}_h, &
			\| \eta_h \|_{\Sigma} &\lesssim \| \bm{u}_h \|_U. \label{eq Poisson h}
	\end{align}
\end{lemma}
\begin{proof}
	We use the same steps as in Lemma~\ref{lem: Poisson} to hierarchically construct $\eta_h \in \Sigma_h$. A concise exposition follows, starting with $d = 0$. For each $i \in I^d$, we use \ref{S: Poisson} to first construct a discrete $\bm{\phi}_{j, h}$ in the trace space $(\bm{n} \cdot \Sigma_h)|_{\Gamma_j}$ for all $j \in \hat{J}_i$ such that
	\begin{subequations}
	\begin{align}
		-\jump{\hat{\epsilon} \bm{\phi}_h}_i 
		&= \hat{\epsilon}_{\max}^2 \bm{u}_{i, h} 
		- \nabla \cdot \epsilon \eta_{i, h}, \\
		\norm{\bm{\phi}_h}_{\Gamma^d} 
		&\lesssim \norm{\hat{\epsilon}_{\max} \bm{u}_h}_{\Omega^d} 
		+ \norm{\bm{\phi}_h}_{\Gamma^{d - 1}} 
	\end{align}
	\end{subequations}
	in which $\nabla \cdot \epsilon \eta_{i, h}$ and $\bm{\phi}_{j, h}$ are understood as zero for $i \in I^0$, $j \in \check{J}_i$. For $i \in I^{d + 1}$, the function $\eta_{i, h}$ is then defined as the bounded, discrete $H(\div, \Omega_{i, h})$-extension from \cite{quarteroni1999domain} (Section 4.1.2). These steps are repeated by incrementing $d$ until $\bm{\phi}_h$ is completely defined on $\Gamma_h$. Finally, we solve a discrete Poisson problem for each $i \in I^3$, in analogy with \eqref{aux problem eta_i}, to complete $\eta_h \in \Sigma_h$. 
	It follows by the same arguments as in Theorem~\ref{Thm: inf-sup} that the constructed $\eta_h$ satisfies \eqref{eq Poisson h}.
\qed
\end{proof}

\begin{lemma} \label{lem: Stokes h}
	Given $r_h \in R_h$, a function $\xi_h \in \Sigma_h$ exists such that
	\begin{align}
			\mathfrak{D} \cdot \epsilon \xi_h &= 0, &
			\Pi_{R_h} \skw \xi_h &= \epsilon r_h, &
			\| \xi_h \|_{\Sigma} &\lesssim \| r_h \|_R. \label{eq Stokes h}
	\end{align}
	with $\Pi_{R_h}$ the $L^2$ projection onto $R_h$.
\end{lemma}
\begin{proof}
	In this proof, we make extensive use of the discrete space $W_h$ from stability requirement \ref{S: W_h}. In particular, we will first introduce $w_h \in W_h$ such that $\mathfrak{D} \times w_h$ controls $r_{i, h}$ for $i \in I^2$. Then, a correction is introduced using \ref{S: stable triplet} in order to control $r_{i, h}$ for $i \in I^3$ as well. For brevity, we omit the subscript $h$ on all variables within this proof.
	
	As in Lemma~\ref{lem: Stokes}, we start with an interface function $\phi$ defined on the trace mesh $\Gamma_h^2$ which serves first as a source function and second as a boundary condition. We proceed according to the following four steps.
	\begin{enumerate}
		\item We consider functions on $\Gamma_h^2$ in the trace space of $(W_h)|_{\Omega^3}$ that vanish at all intersections and extremities. Let us therefore introduce the function space $\Phi_h$ as
		\begin{align}
		 	\Phi_h = \prod_{i \in I^3} \prod_{j \in \check{J}_i} \{ \phi \in (\bm{n}_i \times W_{i, h})|_{\Gamma_j}: \ (\bm{n}_j \cdot \phi)|_{\partial \Gamma_j} = 0 \}.
		\end{align}
		It is important to note that the two instances of $\bm{n}$ in this definition are different. In particular, the former is defined as normal to $\partial \Omega_i$, of which $\Gamma_j$ is a subset, whereas the latter is normal with respect to the boundary $\partial \Gamma_j$. 

		We note that, since $\mathfrak{D} \times W_h \subseteq \Sigma_h$ by \ref{S: W_h}, we have $W_{i, h} \subseteq (H(\curl, \Omega_i))^3$ for $i \in I^3$. In turn, it follows in the discrete setting that $\Phi_{j, h} \in (H(\div, \Gamma_j))^3$ for all $j \in \check{J}_i$, with the divergence tangential to $\Gamma_j$. Using this observation, we let $\phi$ solve the following minimization problem
			\begin{align}
				\min_{\varphi \in \Phi_h} &\ \tfrac{1}{2} \norm{\varphi}_{H(\div, \Gamma^2)}^2 &
			 	& \text{subject to } \ 
				\Pi_{\mathbb{R}_i} (\skw_2 \jump{\varphi}_i + \epsilon^2 r_i) = 0,
				\ \forall i \in I^2.
			\end{align}
		In other words, a finite-dimensional problem is solved for each $i \in I^2$ to obtain a bounded distribution that has an average asymmetry corresponding to $\epsilon^2 r_i$. In particular, we obtain the following two properties
		\begin{subequations}
			\begin{align}
				- \Pi_{\mathbb{R}_i} \skw_2 \jump{\phi}_i &= \Pi_{\mathbb{R}_i} \epsilon^2 r_i, 
				& \forall i &\in I^2, \label{constraints V h} \\
				\| \phi \|_{H(\div, \Gamma^2)} &\lesssim \| \epsilon^2 r \|_{\Omega^2}. \label{Stokes bound theta h}
			\end{align}
		\end{subequations}

		\item We generate $w_i$ for $i \in I^2$ using the distribution $\phi$ from the first step as a source term. Introducing $\Theta_{i, h} = \Xi W_{i, h}$, it follows from \ref{S: W_h} that $\Theta_{i, h} \times R_{i, h}$ is a stable pair for the discretization of the Stokes problem \eqref{eq: Stokes n-1}. Hence, we can find $\theta_i \in \Theta_{i, h}$ such that
		\begin{align} \label{eq: Stokes n-1 h}
				\Pi_{R_{i, h}} \nabla \cdot \theta_i &= (\Pi_{R_{i, h}} - \Pi_{\mathbb{R}_i})(\epsilon^2 r_i + \skw_2 \jump{\phi}), \\
				(\theta_i)|_{\partial \Omega_i} &= 0, \\
			\norm{\theta}_{H^1(\Omega^2)} &\lesssim \norm{\epsilon^2 r}_{\Omega^2} + \norm{\phi}_{\Gamma^2}
			\lesssim \norm{\epsilon^2 r}_{\Omega^2}
		\end{align}
		with $\Pi_{R_{i, h}}$ the $L^2$ projection onto $R_{i, h}$. We set $w_i = \Xi^{-1} \theta_i \in W_{i, h}$ for all $i \in I^2$.

		\item For $i \in I^3$, let $w_i$ be given by any bounded extension of $\phi$ into $H(\curl, \Omega_i)$, i.e. $w_i$ is chosen for all $i \in I^3$ such that
		\begin{subequations}
			\begin{align}
			 	(\bm{n} \times w_i)|_{\Gamma_j} &= \phi_j, & \forall j &\in \check{J}_i \\
			 	\| w \|_{H(\curl, \Omega^3)} &\lesssim \| \phi \|_{H(\div, \Gamma^2)}.
			 \end{align} 
		\end{subequations}

		\item Finally, we gain control of $r_i$ for $i \in I^3$. We recall that by stability condition \ref{S: stable triplet}, the spaces $\Sigma_{i, h} \times \bm{U}_{i, h} \times R_{i, h}$ form a stable triplet for the mixed formulation of elasticity with relaxed symmetry. Considering $\Gamma^2$ as a zero traction boundary condition, we use the inf-sup condition associated to this stability to form a function $\psi_i \in \Sigma_{i, h}$ such that
		\begin{subequations}
			\begin{align}
				\nabla \cdot \psi_i &= 0 \\
				\Pi_{R_{i, h}} \skw_3 \psi_i &= r_i - \Pi_{R_{i, h}} \skw_3 (\nabla \times w_i) \\
				(\bm{n} \cdot \psi_i)|_{\Gamma_j} &= 0, &
				\forall j &\in \check{J}_i \\
				\| \psi \|_{H(\div, \Omega^3)} &\lesssim \| r \|_{\Omega^3} + \| w \|_{H(\curl, \Omega^3)}
			\end{align}
		\end{subequations}
		To complete the mixed-dimensional function $\psi \in \Sigma_h$, we set $\psi_i = 0$ for $i \not \in I^3$.
	\end{enumerate}

	Using the above ingredients, we set $\xi = \psi + \epsilon^{-1} \mathfrak{D} \times w$ and obtain the first two desired properties:
	\begin{align*}
		\mathfrak{D} \cdot \epsilon \xi 
		&= \mathfrak{D} \cdot \epsilon (\psi + \epsilon^{-1} \mathfrak{D} \times w) 
		= 0, \\
		(\Pi_{R_h} \skw \xi)|_{\Omega_i} &= \Pi_{R_{i, h}} (\epsilon^{-1} \skw \mathfrak{D} \times w) \nonumber\\
		&= \Pi_{R_{i, h}} \epsilon^{-1} \skw_2(\nabla^{\perp} w_i - \jump{\bm{n} \times w}_i) \nonumber\\
		&= \Pi_{R_{i, h}} \epsilon^{-1} (\nabla \cdot \theta_i - \skw_2 \jump{\phi}_i)
		= \epsilon r_i, 
		& \forall i &\in I^2, \\
		(\Pi_{R_h} \skw \xi)|_{\Omega_i} &= \Pi_{R_{i, h}} \skw_3 (\psi_i + \nabla \times w_i) 
		= r_i = \epsilon r_i, 
		& \forall i &\in I^3.
	\end{align*}

	The bound now follows by the estimates given in each step:
	\begin{align*}
		\norm{\xi}_\Sigma 
		= \norm{\xi}_{\Omega} +
		\norm{\bm{n} \cdot \xi}_{\Gamma}
		\lesssim \norm{r_h}_R
	\end{align*}
	\qed
	\end{proof}

The two previous theorems provide the sufficient ingredients to show stability of the discretization, formally presented in the following theorem.

\begin{theorem}[Stability] \label{thm: stability}
	If the discrete spaces satisfy the conditions \ref{S: conformity}-\ref{S: stable triplet}, then the resulting mixed finite element method is stable, i.e. a unique solution exists satisfying the bound
	\begin{align}
		\norm{\sigma_h}_\Sigma + \norm{\bm{u}_h}_U + \norm{r_h}_R \lesssim 
		\normL{\epsilon \bm{f}} + \norm{\hat{\epsilon}_{\max} \bm{g}_u}_{H^{\frac{1}{2}}(\partial_u \Omega)}
	\end{align}
	with a constant independent of the grid size $h$.
\end{theorem}
\begin{proof}
	Using Theorems \ref{thm: ellip_h} and \ref{Thm: inf-sup_h}, this result follows from standard saddle point theory \cite{Boffi}.
\qed
\end{proof}

\subsection{Convergence} \label{sub: Convergence}

By consistency of the discretized problem \eqref{weakform_h} with respect to the continuous formulation \eqref{weakform} and stability from Theorem~\ref{thm: stability}, we have shown that the proposed mixed finite element discretization is convergent. In turn, this section is devoted to obtaining the rates of convergence through a priori error estimation.

Let $\Pi_{R_h}$ and $\Pi_{U_h}$ be the $L^2$-projection operators onto the finite element spaces $R_h$ and $\bm{U}_h$. Under the assumption of sufficient regularity, we introduce the canonical projection operator $\Pi_{\Sigma_h}$ such that the commutativity properties hold for $\sigma \in \Sigma$:
\begin{align*}
	\Pi_{U_{i, h}} \nabla \cdot \sigma_i &= 
	\nabla \cdot \Pi_{\Sigma_{i, h}} \sigma_i, 
	\\
	\Pi_{U_{i, h}} (\bm{n} \cdot \sigma_\hatj)|_{\Gamma_j} &= 
	(\bm{n} \cdot \Pi_{\Sigma_{\hatj, h}} \sigma_\hatj)|_{\Gamma_j}, 
	& \forall i &\in I, j \in \hat{J}_i.
\end{align*}
A direct consequence of these two properties is that for sufficiently regular $\sigma \in \Sigma$, we have the commuting property
\begin{align} \label{eq: commutativity}
	\Pi_{U_h} \mathfrak{D} \cdot \sigma &= 
	\mathfrak{D} \cdot \Pi_{\Sigma_h} \sigma.
\end{align}

Using $\| \cdot \|_{\rho, \Omega}$ as short-hand notation for the $H^\rho(\Omega)$-norm, the projection operator $\Sigma_h$ has the following approximation properties for $0 \le \rho \le k_i + 1$:
	\begin{subequations} \label{eqs proj est}
	\begin{align}
			\| (I - \Pi_{\Sigma_{i, h}}) \sigma_i \|_{\Omega_i} &\lesssim h^\rho ~ \| \sigma_i \|_{\rho,\Omega_i}, 
			& i &\in \cup_{d = 1}^n I^d, 
	\end{align}
The maximal rate is given by $k_i = (k + 1) + (n - d_i)$ for the full spaces (see Table~\ref{tab: W_h}) and $k_i = k + 1$ for the reduced spaces (Table~\ref{tab: W_h reduced}). Additionally, we have the following properties for $0 \le \rho \le k_i$:
		\begin{align}
			\| \nabla \cdot ((I - \Pi_{\Sigma_{i, h}}) \sigma_i) \|_{\Omega_i} &\lesssim h^\rho ~ \| \nabla \cdot \sigma_i \|_{\rho,\Omega_i},
			& i &\in \cup_{d = 1}^n I^d, 
			\\ 
			\| \bm{n} \cdot ((I - \Pi_{\Sigma_{\hatj, h}}) \sigma_\hatj) \|_{\Gamma_j} &\lesssim h^\rho ~ \| \bm{n} \cdot \sigma_\hatj \|_{\rho,\Gamma_j}, 
			& i &\in \cup_{d = 0}^{n - 1} I^d, 
			\ j \in \hat{J}_i,
			\label{proj est sig n}
			\\
			\| (I - \Pi_{U_{i, h}}) \bm{u}_i \|_{\Omega_i} &\lesssim h^\rho ~ \| \bm{u}_i \|_{\rho,\Omega_i}, 
			& i &\in I, 
			\label{proj est u}\\
			\| (I - \Pi_{R_{i, h}}) r_i \|_{\Omega_i} &\lesssim h^\rho ~ \| r_i \|_{\rho,\Omega_i}, 
			& i &\in \cup_{d = 2}^n I^d, 
		\end{align}
	\end{subequations}
	Note that for $\Omega^0$ and $\Gamma^0$, the projection operators correspond to the identity operator which makes \eqref{proj est sig n} and \eqref{proj est u} trivial there.

\begin{theorem}[Convergence]
	Given $(\sigma, \bm{u}, r)$ as the solution to \eqref{weakform}. Under the assumption of sufficient regularity, then the scheme converges optimally, i.e. the finite element solution $(\sigma_h, \bm{u}_h, r_h)$ to \eqref{weakform_h} satisfies the following estimate
	\begin{align*}
		\| \sigma - \sigma_h \|_\Sigma
		+ \norm{ \bm{u} - \bm{u}_h }_{U} 
		&+ \norm{ r - r_h }_{R}
		\lesssim
			  \nonumber\\
		\sum_{i \in I}  h^{k_i} \Big( &h \norm{\sigma}_{k_i + 1, \Omega_i}
			+ \norm{ \hat{\epsilon}_{\max}^{-1} \mathfrak{D} \cdot \epsilon \sigma}_{k_i, \Omega_i} \\
			&+ \sum_{j \in \hat{J}_i} \norm{\bm{n} \cdot \sigma }_{k_i, \Gamma_j} 
			+ \norm{ \hat{\epsilon}_{\max} \bm{u} }_{k_i, \Omega_i} 
			+ \norm{ \epsilon r }_{k_i, \Omega_i} \Big)
	\end{align*}
\end{theorem}
\begin{proof}
	We restrict our choice of test functions $(\tau_h, \bm{v}_h, s_h)$ to the finite element spaces and subtract the discrete equations \eqref{weakform_h} from the continuous equations \eqref{weakform}. We then obtain
	\begin{align*}
		a(\sigma - \sigma_h; \tau_h)
		+ b(\tau_h; \bm{u} - \bm{u}_h, r - r_h)
		- b(\sigma - \sigma_h; \bm{v}_h, s_h) &= 0.
	\end{align*}

	The projection operators from \eqref{eqs proj est} are then used to project the true solution onto the finite element spaces. Introducing $\sigma_\Pi := \Pi_{\Sigma_h} \sigma$, $\bm{u}_\Pi := \Pi_{U_h} \bm{u}$, and $r_\Pi := \Pi_{R_h} r$, we rewrite the above equation to
	\begin{align} \label{eq: Diff disc cont}
		a(\sigma_\Pi - \sigma_h; \tau_h)
		&+ b(\tau_h; \bm{u}_\Pi - \bm{u}_h, r_\Pi - r_h)
		- b(\sigma_\Pi - \sigma_h; \bm{v}_h, s_h) \nonumber\\
		&= a(\sigma_\Pi - \sigma; \tau_h)
		+ b(\tau_h; \bm{u}_\Pi - \bm{u}, r_\Pi - r)
		- b(\sigma_\Pi - \sigma; \bm{v}_h, s_h)
	\end{align}

	Let $\tau_{b, h}$ be the discrete stress from the construction in Theorem~\ref{Thm: inf-sup_h} with the following properties
	\begin{subequations}
		\begin{align}
			b(\tau_{b, h}; \bm{u}_\Pi - \bm{u}_h, r_\Pi - r_h) 
			&= \norm{\bm{u}_\Pi - \bm{u}_h}_{U}^2
			+ \norm{r_\Pi - r_h}_{R}^2 \\
			\| \tau_{b, h} \|_\Sigma
			&\lesssim \norm{\bm{u}_\Pi - \bm{u}_h}_{U}
			+ \norm{r_\Pi - r_h}_{R} \label{bound on tau_bh}
		\end{align}
	\end{subequations}

	The discrete test functions are then chosen to be
		\begin{align}
			\tau_h &:= \sigma_\Pi - \sigma_h + \delta \tau_{b, h}, &
			\bm{v}_h &:= \bm{u}_\Pi - \bm{u}_h, &
			s_h &:= r_\Pi - r_h,
		\end{align}
	with $\delta > 0$ a constant to be determined later. Substituting this choice of functions into the left-hand side of \eqref{eq: Diff disc cont} gives us
	\begin{align*}
		a(\sigma_\Pi - \sigma_h; \tau_h)
		&= a(\sigma_\Pi - \sigma_h; \sigma_\Pi - \sigma_h) 
		+ a(\sigma_\Pi - \sigma_h; \delta \tau_{b, h}), \nonumber\\
		b(\tau_h; \bm{u}_\Pi - \bm{u}_h, r_\Pi - r_h) \nonumber\\
		- b(\sigma_\Pi - \sigma_h; \bm{v}_h, s_h)
		&= b(\delta \tau_{b, h}; \bm{u}_\Pi - \bm{u}_h, r_\Pi - r_h) \nonumber\\
		&= \delta (\norm{ \bm{u}_\Pi - \bm{u}_h }_U^2 
		+ \norm{ r_\Pi - r_h }_R^2).
	\end{align*}
	Next, due to \eqref{eq: commutativity} and \ref{S: Poisson}, we have $\mathfrak{D} \cdot \epsilon(\sigma_\Pi - \sigma_h) = \Pi_{U_h} \mathfrak{D} \cdot \epsilon(\sigma - \sigma_h) = 0$. Hence, the coercivity of $\mathfrak{A}$ from \eqref{Bounds A} allows us to bound $a(\sigma_\Pi - \sigma_h; \sigma_\Pi - \sigma_h)$ from below by $\norm{\sigma_\Pi - \sigma_h}_\Sigma^2$. We then obtain the following bound with respect to the right-hand side of \eqref{eq: Diff disc cont}:
	\begin{align*}
		\norm{\sigma_\Pi - \sigma_h}_\Sigma^2 
		&+ \delta (\norm{ \bm{u}_\Pi - \bm{u}_h }_U^2 
		+ \norm{ r_\Pi - r_h }_R^2)
		\lesssim \nonumber\\
		&- a(\sigma_\Pi - \sigma_h; \delta \tau_{b, h})
		+ a(\sigma_\Pi - \sigma; \sigma_\Pi - \sigma_h + \delta \tau_{b, h}) \nonumber\\
		&+ b(\sigma_\Pi - \sigma_h + \delta \tau_{b, h}; \bm{u}_\Pi - \bm{u}, r_\Pi - r) \nonumber\\
		&- b(\sigma_\Pi - \sigma; \bm{u}_\Pi - \bm{u}_h, r_\Pi - r_h).
	\end{align*}
	Next, we use the continuity of the forms $a$ and $b$ from Lemma~\ref{lem: continuity} to bound the right-hand side further
	\begin{align*}
		\ldots 
		\lesssim &\ 
		\norm{ \sigma_\Pi - \sigma_h }_\Sigma
		\delta \norm{\tau_{b, h}}_\Sigma \nonumber\\
		&+ (\norm{\sigma_\Pi - \sigma_h}_\Sigma + \delta \norm{\tau_{b, h}}_\Sigma)
		(\norm{ \sigma_\Pi - \sigma }_\Sigma + \norm{\bm{u}_\Pi - \bm{u}}_U + \norm{r_\Pi - r}_R) \nonumber\\
		&+ 
		(\norm{ \bm{u}_\Pi - \bm{u}_h}_U
		+ \norm{ r_\Pi - r_h}_R)
		\norm{ \sigma_\Pi - \sigma}_\Sigma
		.
	\end{align*}		
	An application of Young's inequality and rearranging terms then gives
	\begin{align*}
		\norm{\sigma_\Pi - \sigma_h}_\Sigma^2 
		&+ 
		\delta (\norm{ \bm{u}_\Pi - \bm{u}_h }_U^2 
		+ \norm{ r_\Pi - r_h }_R^2) \nonumber\\
		&\lesssim
		\delta^2 \norm{\tau_{b, h}}_\Sigma^2
		+ (1 + \delta^{-1})\norm{ \sigma_\Pi - \sigma }_\Sigma^2 
		+ \norm{\bm{u}_\Pi - \bm{u}}_U^2 
		+ \norm{r_\Pi - r}_R^2.
	\end{align*}
	Using \eqref{bound on tau_bh} and setting $\delta$ sufficiently small leads us to 
	\begin{align*}
		\| \sigma_\Pi - \sigma_h \|_\Sigma^2
		&+ \norm{ \bm{u}_\Pi - \bm{u}_h }_{U}^2 
		+ \norm{ r_\Pi - r_h }_{R}^2 \nonumber\\
		&\lesssim \| \sigma_\Pi - \sigma \|_\Sigma^2
		+ \norm{ \bm{u}_\Pi - \bm{u} }_{U}^2 
		+ \norm{ r_\Pi - r }_{R}^2.
	\end{align*}
	With the triangle inequality, we thus obtain the estimate
	\begin{align*}
		\| \sigma - \sigma_h \|_\Sigma
		&+ \norm{ \bm{u} - \bm{u}_h }_{U} 
		+ \norm{ r - r_h }_{R} \\
		&\lesssim \| \sigma_\Pi - \sigma \|_\Sigma
		+ \norm{ \bm{u}_\Pi - \bm{u} }_{U} 
		+ \norm{ r_\Pi - r }_{R}.
	\end{align*}
	An application of the approximation properties \eqref{eqs proj est} then finishes the proof.
	\qed
	\end{proof}

\bibliographystyle{spmpsci}
\bibliography{elasbib}

\end{document}